\documentclass[11pt]{article}

\usepackage{graphicx}

\usepackage[colorlinks=true]{hyperref}
\hypersetup{urlcolor=blue, citecolor=red, linkcolor=blue}
\usepackage{amsmath,amssymb}
\usepackage{bm}
\usepackage{enumerate}
\usepackage{amsthm}
\usepackage[all,2cell]{xy}
\usepackage{mathrsfs}
\usepackage{mathtools}
\mathtoolsset{showonlyrefs}
\usepackage{tikz-cd}
\usepackage{xcolor}
\usepackage{physics}

\usetikzlibrary{cd}

\usepackage[a4paper, left=23mm, right=23mm, top=30mm, bottom=30mm]{geometry}

\usepackage{marginnote}

\usepackage{imakeidx}
\makeindex[intoc]

\usepackage[nottoc]{tocbibind}

\theoremstyle{plain}
\newtheorem{theorem}{Theorem}[section]
\newtheorem{corollary}[theorem]{Corollary}
\newtheorem{proposition}[theorem]{Proposition}
\newtheorem{lemma}[theorem]{Lemma}
\newtheorem{example}[theorem]{Example}

\theoremstyle{definition}
\newtheorem{definition}[theorem]{Definition}
\newtheorem{remark}[theorem]{Remark}

\newcommand\restr[2]{{
  \left.\kern-\nulldelimiterspace 
  #1 
  \right|_{#2} 
}}

\newcommand{\R}{\mathbb{R}}
\renewcommand{\S}{\mathbb{S}}
\renewcommand{\d}{\mathrm{d}}

\newcommand{\df}{\Omega}
\newcommand{\Cinfty}{\mathscr{C}^\infty}
\newcommand{\T}{\mathrm{T}}
\newcommand{\cT}{\mathrm{T}^\ast}

\newcommand{\X}{\mathfrak{X}}

\newcommand{\C}{\mathcal{C}}

\newcommand{\parder}[2]{\frac{\partial #1}{\partial #2}}
\newcommand{\dparder}[2]{\dfrac{\partial #1}{\partial #2}}

\DeclareMathOperator{\Ima}{Im}
\DeclareMathOperator{\inn}{\iota}

\DeclareMathOperator{\rk}{rank}

\DeclareMathAlphabet{\mathpzc}{OT1}{pzc}{m}{it}
\def\d{\mathrm{d}}

\title{{\sffamily 
On Darboux theorems for geometric structures\\\vskip .2cm induced by closed forms}}

\author{{\sffamily 
$^a$Xavier Gràcia%
\thanks{e-mail:
   xavier.gracia@upc.edu \ ORCID: 0000-0003-1006-4086}\ ,\
$^b$Javier de Lucas%
\thanks{e-mail:
   javier.de.lucas@fuw.edu.pl \ ORCID: 0000-0001-8643-144X}\ ,\
$^c$Xavier Rivas%
\thanks{e-mail:
   xavier.rivas@unir.net \ ORCID: 0000-0002-4175-5157}\ ,\
$^a$Narciso Román-Roy%
\thanks{e-mail:
   narciso.roman@upc.edu \ ORCID: 0000-0003-3663-9861}
}
\\[2ex]
\normalsize\itshape\sffamily
$^a$Department of Mathematics,
\\ \normalsize\itshape\sffamily 
Universitat Polit\`ecnica de Catalunya, Barcelona, Spain.
\\[1ex]
\normalsize\itshape\sffamily
$^b$UW Institute for Advanced Studies,\\
\normalsize\itshape\sffamily
Department of Mathematical Methods in Physics, University of Warsaw, Warszawa, Poland.
\\[1ex]
\normalsize\itshape\sffamily
$^c$Escuela Superior de Ingenier\'{\i}a y Tecnolog\'{\i}a,\\
\normalsize\itshape\sffamily
Universidad Internacional de La Rioja, Logro\~no, Spain.
\\[1ex]
}

\begin{document}

\maketitle

\parskip=4pt
\vskip-2cm
\begin{abstract}

This work reviews the classical Darboux theorem for symplectic, presymplectic, and cosymplectic manifolds (which are used to describe regular and singular mechanical systems), and certain cases of multisymplectic manifolds, and extends it in new ways to $k$-symplectic and $k$-cosymplectic manifolds (all these structures appear in the geometric formulation of first-order classical field theories). Moreover, we discuss the existence of Darboux theorems for classes of precosymplectic, $k$-presymplectic, $k$-precosymplectic, and premultisymplectic manifolds, which are the geometrical structures underlying some kinds of singular field theories. Approaches to Darboux theorems based on flat connections associated with geometric structures are given, while new results on polarisations for ($k$-)(pre)(co)symplectic structures arise.

\end{abstract}

\vskip -0.1cm
\noindent {\bf MSC\,2020:} \textsl{Primary:}
53C15, 
53C12. 
\textsl{ Secondary:}
53D05, 
53C10. 

 



{\setcounter{tocdepth}{2}
\def\baselinestretch{1}
\small
\def\addvspace#1{\vskip 1pt}
\parskip 0pt plus 0.1mm
\tableofcontents
}


\section{Introduction}

Since its very origins, differential geometry has been applied to many branches of mathematical physics to study different kinds of physical systems, and it has led to many developments. Symplectic geometry, namely the study of closed non-degenerate two-forms, the so-called {\it symplectic forms}, was one of the first areas of differential geometry to be introduced. 
Symplectic geometry has its origins in the study of celestial mechanics \cite{Mar_09}, it has a relevant role in classical mechanics \cite{AMR_88,Can_01}, and it has inspired the development of many other useful geometric structures with relevant applications \cite{AM_78,AG_00,LSV_15,Kij_73,KT_79,LM_87}.

One of the fundamental results in symplectic geometry is the {\it Darboux theorem},
which describes the local structure of finite-dimensional symplectic manifolds \cite{Dar_82}. Roughly speaking, the Darboux theorem states that a symplectic form can be locally written as a differential form with constant coefficients of a particular type, namely as the canonical symplectic form of a cotangent bundle in adapted coordinates, $\omega = \d q^i\wedge\d p_i$ \cite{AM_78,Lee_13}. There exist several types of infinite-dimensional symplectic manifolds, and some of them do not admit a Darboux theorem \cite{Mar_72}. Hereafter, we focus on finite-dimensional manifolds, unless otherwise stated.
   
The Darboux theorem can be proved in different ways \cite{AM_78,BCGGG_91},
and its proof can be extended to presymplectic forms, namely closed two-forms of constant rank \cite{CP_86}.
It is well-known that symplectic and presymplectic forms describe
the phase spaces for autonomous regular and singular dynamical systems in mechanics.
For non-autonomous mechanical systems, the suitable structures are the so-called
{\it cosymplectic} and {\it precosymplectic} manifolds \cite{CLM_94,LMM_96}.

As a preliminary goal, this paper reviews the theory of Darboux theorems for symplectic and presymplectic manifolds, and it analyses their relation to the so-called flat compatible symplectic and presymplectic connections \cite{BCGRS_06,GRS_98,Vai_00}. Connections are hereafter assumed to be linear and torsionless, being the second condition usual in the literature and a key to the description of certain features of the differential forms and integrable distributions to be studied in this work. We also provide proofs of the Darboux theorems for cosymplectic and precosymplectic manifolds. The Darboux theorem for precosymplectic structures is assumed in many references but its proof seems to be absent in the literature \cite{CLM_94}.

To achieve a geometrical covariant description of (first-order) classical field theories,
the above-mentioned structures have been generalised in several ways.
The simplest ones are the so-called {\it $k$-symplectic manifolds}, introduced by A. Awane 
\cite{Awa_92, AG_00} and used later by M. de Le\'{o}n \textit{et al.} 
\cite{LMOS_97, LMS_88, LMS_88a} and L.K. Norris \cite{MN_00, Nor_93} for describing first-order classical field theories.
They coincide with the {\it polysymplectic manifolds}
described by  G.C. G\"unther \cite{Gun_87}
(although these last ones are different from those introduced
by G. Sardanashvily~et al.~\cite{GMS_97,Sar_95} 
and I.V. Kanatchikov \cite{Kan_98}, that are also called {\it polysymplectic}).
This structure is used to give a geometric description of regular field theories
whose Lagrangian and/or Hamiltonian functions, in a local description,
do not depend on the space-time coordinates (or the analogous).
In the degenerate case we use {\it $k$-presymplectic structures},
which allow us to describe the corresponding field theories given by singular Lagrangian functions \cite{GMR_09}. It is worth stressing that there exist several ways of defining $k$-presymplectic manifolds, some of which have apparently been proposed and studied in the present work for the first time. 

A natural extension of $k$-symplectic manifolds are {\it $k$-cosymplectic manifolds},
which enable one to generalise the cosymplectic description of
non-autonomous mechanical systems to regular field theories
whose Lagrangian and/or Hamiltonian functions, in a local description,
depend on the space-time coordinates (or the analogous) \cite{LMORS_98,LMS_01}.  As previously, the singular case of these theories leads to the introduction of {\it $k$-precosymplectic manifolds}, which can be defined in different manners, as shown in this paper and studied in previous works \cite{GRR_20}.

The Darboux theorem was generalised and proved for $k$-symplectic manifolds in \cite{Awa_92,LMS_88a,FG_13} and for $k$-cosymplectic manifolds in \cite{LMORS_98}. The Darboux theorem plays a relevant role in these theories since, for instance, it significantly simplifies the proofs of many results \cite{RRSV_11}. In this work, we provide a (as far as we know) new  Darboux theorem for $k$-symplectic and $k$-cosymplectic linear spaces. We also provide new proofs for the Darboux theorems for $k$-symplectic and $k$-cosymplectic manifolds. Our proofs reveal new properties of such types of manifolds concerning the properties of their Lagrangian submanifolds. In particular, new details about the existence of the hereafter called polarisations for $k$-symplectic and $k$-cosymplectic manifolds are obtained. Moreover, classical proofs of the $k$-symplectic manifold rely on coordinates and special, rather lengthy calculations \cite{AG_00}. Others are focused on connections and give indirect proofs \cite{FG_13}. Meanwhile, our proof of the $k$-symplectic Darboux theorem is intrinsic and short. Moreover, our proof could have been made shorten by relying on known results, but we decided to give a full explanation of all the structures and results involved, which made it longer than strictly needed to prove the canonical form of $k$-symplectic manifolds.

Darboux theorems for $k$-symplectic manifolds are closely related to the notion of polarisation \cite{Awa_92}. This means that we search for coordinates where the $k$-(co)symplectic structures take a form with constant coefficients of a particular type. Nevertheless, one could find new coordinates where the $k$-(co)symplectic forms would take constant coefficients of another different type. This would potentially lead to Darboux coordinates of other types.

Moreover, one may try to find coordinates to put the differential forms of a $k$-symplectic structure on a canonical manner. This leads to the existence of a certain type of associated  distribution. Notwithstanding, Darboux coordinates can be defined to additionally put a basis of the distribution in a particular manner. It is worth noting that, in the case of $k$-symplectic manifolds, the conditions to obtain Darboux coordinates putting the associated differential forms into canonical form ensure that there exists a canonical basis of the distribution too. Meanwhile, our $k$-cosymplectic Darboux theorem shows that the conditions to put the differential forms associated with a $k$-cosymplectic manifold are different from those needed to ensure also a canonical form for a basis of the associated distribution. Moreover, our analysis  also sheds some light in the existence of Darboux coordinates for $k$-cosymplectic manifolds, and it complements the results given in previous works \cite{GM_23,Mer_97}. In particular, it is worth noting that Theorem II.4 and Theorem 5.2.1 in \cite{LMORS_98,Mer_97} can be slightly misleading, as part of the assumptions needed to prove such theorems are only described in Remark 2.5 and Note 5.2.1, after them, respectively. 


Then, we study $k$-presymplectic manifolds. These structures appear as a side problem in $k$-symplectic or multisymplectic theories \cite{FG_13}. We here prove that the very definition of a $k$-presymplectic manifold can be set in different ways, depending on the features that we want them to have, e.g. to fit the analysis of systems we are dealing with. Some of these notions of $k$-presymplectic manifold do not admit a Darboux theorem of the initially expected form, even for the linear $k$-presymplectic cases.  Then, we study some different possible definitions of $k$-presymplectic manifolds, and we provide some counterexamples showing that a Darboux theorem does not need to exist for them. This is quite unexpected, as it was previously assumed that Darboux theorems must be satisfied for them. It is worth noting that the authors in \cite{GMR_09} remark that the existence of a Darboux theorem for $k$-presymplectic manifolds is an open problem, although they skip this by giving intrinsic proofs of their results.  The same happens when we consider $k$-precosymplectic manifolds \cite{GRR_20} in order to deal with non-autonomous field theories described by singular Lagrangian functions.

As in the $k$-presymplectic manifolds case, one has the same type of problems and similar solutions are given. A Darboux theorem for precosymplectic manifolds has been provided. Although this result has been used in the literature \cite{CLM_94,LR_89}, it seems that a proof is missing. More generally, we have provided definitions of $k$-pre(co)symplectic manifolds admitting Darboux theorems. This gives an alternative approach to previous point-wise and local Darboux theorems in \cite{FG_13} for the $k$-presymplectic case. Moreover, our point-wise and local $k$-precosymplectic Darboux theorems seem to be new. Note also that $k$-precosymplectic manifold do not have canonically defined Reeb vector fields. This causes that the Darboux theorems may not involve the existence of a basis of them given in a canonical form. Moreover, the distribution defined to put the differential forms of $k$-precosymplectic manifolds in a canonical manner does not admit a canonical basis unless additional conditions are given.

Finally, we have the {\it multisymplectic manifolds}
first introduced by J. Kijowski, W.M. Tulczyjew, and other authors
\cite{Gar_74,GS_73,Kij_73,KT_79}, which constitute one of the most generic structures
for studying the behaviour of Lagrangian and Hamiltonian field theories (see \cite{Rom_09} and references therein).
Nevertheless, although there are some partial results
\cite{Mar_88},
a Darboux-type theorem for multisymplectic manifolds in general is not known.
In particular, a class of multisymplectic manifolds with a local structure defined by Darboux type
coordinates was characterised in \cite{CIL_99}, and
certain kinds of multisymplectic manifolds admitting Darboux coordinates
have been described in \cite{LMS_03},
giving a sufficient condition that guarantees the existence of Darboux charts.


While studying the different geometric structures, we analyse the existence of linear connections compatible with them. Some of our results are known, see for instance
symplectic connections \cite{BCGRS_06,GRS_98}, $k$-symplectic connections \cite{Bla_09}, $k$-cosymplectic \cite{Mer_97}, multisymplectic and polysymplectic connections \cite{FY_13}. On the other hand, some of the connections compatible with these and other structures are here proposed. Moreover, we here review the subject and serves as a reference point for further works.

The structure of the paper goes as follows. Section \ref{sec2:geomset} reviews the Darboux theorems for symplectic, presymplectic, cosymplectic, and precosymplectic structures and its relation to flat compatible connections. In Section \ref{sec3:k-symp}, we provide a new proof of the Darboux theorem for $k$-symplectic manifolds, which is simpler than previous proofs \cite{Awa_92,AG_00,Riv_22}. We also discuss the existence of Darboux coordinates in $k$-presymplectic manifolds and show that in order to ensure its existence, some very restrictive hypothesis are required. Section \ref{sec4:k-cosymp} is devoted to study the existence of Darboux coordinates in $k$-(pre)cosymplectic manifolds. We give a new proof of the Darboux theorem for $k$-cosymplectic manifolds \cite{LMORS_98}. We also see that it is not possible to ensure the existence of Darboux coordinates, unless some additional conditions are imposed. In Section \ref{sec5:multisymp} we review the existing results on Darboux coordinates for (pre)multisymplectic structures. Some new results on this topic are presented. Finally, Section \ref{sec:outlook} summarises our results and gives some hints on future work. It is worth noting that we explain how the use of flat connections with torsion may be used to study geometric structures related to differential forms that are not closed, such as contact ones. This will be the topic of another paper.

\section{Darboux theorems, flat connections, and symplectic-like structures}\label{sec2:geomset}

Let us set some general assumptions to be used throughout this work.
It is hereafter assumed that all structures are smooth. Manifolds are real, Hausdorff, connected, second countable, and finite-dimensional. 
Differential forms are assumed to
have constant rank, unless otherwise stated. Sum over crossed repeated indices is understood. Sometimes, the summation sign, $\Sigma$, will be used to make clear the range of the indexes we are summing over. All our considerations are local,
to avoid technical problems concerning the global existence of quotient manifolds and similar issues. 
Hereafter, $M$ and $Q$ are assumed to be manifolds, $\mathfrak{X}(M)$ and $\Omega^k(M)$ stand for the $\Cinfty(M)$-modules of vector fields and differential $k$-forms on $M$. Moreover, connections are assumed to be linear and torsion-free.

More particularly, this section reviews  (co)symplectic and (co)presymplectic manifolds and give the corresponding Darboux theorems. It also analyses the relation of Darboux theorems with compatible flat connections. We will also introduce the concept of characteristic distribution, as it will play an important role when generalising,  in Sections \ref{sec3:k-symp} and \ref{sec4:k-cosymp}, the results of this section.

\subsection{Symplectic and presymplectic manifolds}

This section reviews the definition of symplectic and presymplectic manifolds, and it also analyses their corresponding Darboux theorems. In the context of presymplectic manifolds, we recall the definition of their characteristic distributions. For symplectic and presymplectic manifolds, the relation between compatible connections and Darboux coordinates is studied.

\begin{definition}\label{dfn-symplectic-manifold}
    A {\it symplectic manifold} is a pair $(M,\omega)$, where $M$ is a manifold and $\omega$ is a closed differential two-form on $M$ that is {\it non-degenerate}, i.e. the contraction $\inn_X\omega=0$, for a vector field $X$ on $M$, if and only if $X=0$. 
\end{definition}

The canonical model for symplectic manifolds is the cotangent bundle of a  manifold $Q$, namely $(\cT Q,\omega_Q)$, where $\omega_Q\in\Omega^2(\cT Q)$ is the canonical \textit{symplectic two-form} in $\cT Q$, whose local expression in adapted coordinates $\{q^i,p_i\}$ of $\cT Q$ on their associated coordinated open subset of $\cT Q$ is $\omega_Q = \d q^i\wedge\d p_i$.

A symplectic manifold $(M,\omega)$ gives rise to the \textit{musical} (vector bundle) isomorphism $\flat: \T Q\rightarrow \cT Q$ and its inverse $\sharp:\cT Q\to \T Q$ naturally induced by the $\Cinfty(M)$-module isomorphisms
\begin{align*}
    \flat \colon \X(M) & \longrightarrow \Omega^1(M)\\
    X &\longmapsto \inn_X\omega
\end{align*}
and $\sharp = \flat^{-1}$. Note that a vector bundle morphism $\flat$ can be defined for every two-form $\omega$, but $\sharp$ only exists when $\flat$ is invertible or, equivalently, $\omega$ is non-degenerate.

\begin{definition}
    Let $(M,\omega)$ be a symplectic manifold. Given a distribution $D\subset\T M$, the \textit{symplectic orthogonal} of $D$ is defined by $D^\perp = \coprod_{x\in M}D^\perp_x$, where
    $$ D_x^\perp = \{v\in\T_xM\mid \omega_x(v,u) = 0,\, \forall u\in D_x\}\,, $$
    and $\displaystyle\coprod_{x\in M}D_x^\perp$ stands for the disjoint sum of all $D_x^\perp$ over $x\in M$.
\end{definition}

Symplectic orthogonals allow us to introduce several types of submanifolds of symplectic manifolds.

\begin{definition}
    Let $(M,\omega)$ be a symplectic manifold and consider a submanifold $N\subset M$. Then,
    \begin{itemize}
        \item the submanifold $N$ is said to be \textit{isotropic}, if $\T N\subset\T N^\perp$.
        \item the submanifold $N$ is \textit{coisotropic}, if $\T N^\perp \subset \T N$.
        \item the submanifold $N$ is \textit{Lagrangian} if it is isotropic and coisotropic, namely if $\T N^\perp = \T N$. Lagrangian submanifolds are also called \textit{maximally isotropic} and then $2\dim N = \dim M$.
    \end{itemize}
\end{definition}

\begin{definition}
    Two symplectic manifolds $(M_1,\omega_1)$ and $(M_2,\omega_2)$ are \textit{symplectomorphic} if there exists a diffeomorphism $\phi: M_1\to M_2$ such that $\phi^*\omega_2 = \omega_1$.
\end{definition}

The classical \textit{Darboux theorem} states that every symplectic manifold is locally {\it symplectomorphic}
to a cotangent bundle endowed with its canonical symplectic structure \cite{AM_78,BCGGG_91}. The Darboux theorem was initially proved by Darboux \cite{Dar_82}, but its modern standard proof relies on the so-called Moser's trick \cite[Theorem 22.13]{Lee_13}. The statement of the Darboux theorem for symplectic manifolds goes as follows.

\begin{theorem}\label{thm-darboux-symplectic}
    Let $(M, \omega)$ be a symplectic manifold. Then, for every $x\in M$, there exist local coordinates $\{ q^i, p_i\}$ around $x$ where $\omega = \d q^i\wedge\d p_i$.
\end{theorem}

Note that the Darboux theorem amounts to saying that there exist, on a neighbourhood of any $x\in M$, two foliations by transversal Lagrangian submanifolds. 

In infinite-dimensional manifolds, one can still define a symplectic form, but the induced musical morphism $\flat:\X(M)\rightarrow \Omega^1(M)$ is, in general, only injective. This gives rise to the so-called {\it weak symplectic manifolds}. Meanwhile, if $\flat:\mathfrak{X}(M)\rightarrow \Omega^1(M)$ is an isomorphism, then the symplectic manifold is said to be a {\it strong symplectic manifold}. There exists no Darboux theorem for general weak symplectic manifolds \cite{Mar_72}. Nevertheless, by requiring appropriate additional conditions, an analogue can be derived \cite{Tro_76}.


Let us introduce the notion of symplectic connection \cite{BCGRS_06,Sle_70}. Note that the torsion-free assumption is common in the literature, and it is a key for certain results to be developed. Indeed, we will show in Section \ref{sec:outlook} that skipping it leads to a more general theory, but more involved, inappropriate, and unnecessary for our present work. Apart from that last comment to be given in the conclusions of this work, all connections are assumed to be  torsion-free, unless otherwise stated.

\begin{definition}\label{def:symplectic-connection}
A {\it symplectic connection} on a symplectic manifold $(M,\omega)$ is a  connection $\nabla$ on $M$ such that $\nabla\omega = 0$.  The symplectic form $\omega$ is said to be {\it parallel} relative to $\nabla$.
\end{definition}
Every symplectic manifold admits local symplectic connections. 
Indeed, as a consequence of the Darboux theorem, one can construct a local symplectic flat connection for $(M,\omega)$ around every $x\in M$ by assuming that its Christoffel symbols vanish in some Darboux coordinates defined around $x$. 
Such a connection is flat and torsion-free. In general,  flat symplectic connections cannot be globally defined, as it is known that the curvature of a connection is linked to the topology of the manifold where it is defined on. Milnor proved and surveyed in \cite{Mil_58} several results on the existence of connections on the tangent bundle to a manifold. For instance, the tangent bundle to a closed and oriented surface of genus $g$ has no flat connection if $|2-2g|\geq g$. The sphere $\S^2$ has zero genus. Hence, there is no flat connection on $\S^2$, which admits a natural symplectic structure $\omega_{\S^2}$ that can be defined by considering $\S^2$ embedded in $\R^3$ and setting
$$
(\omega_{\S^2})_x(v_x,v'_x)=\langle x,v_x\times v'_x\rangle\,,\qquad \forall x\in \mathbb{R}^3\,,\quad\forall v_x,v'_x\in \T_x\S^2\subset \T_x\R^3\simeq \R^3\,,
$$
where tangent vectors at $x\in\S^2$ are naturally understood as vectors in $\mathbb{R}^3$ and, hence, their vector products are defined. Meanwhile, $\langle \cdot,\cdot\rangle$ stands for the natural scalar product in $\mathbb{R}^3$.  
Let us recall that, if a connection is flat, the parallel transport of a tangent vector along a path contained in a small open set $U$ does not depend on the path. Thus, a basis $\{e_1,\ldots,e_n\}$ of $\T_xM$ gives rise, by parallel transport, to a family of vector fields $X_1,\ldots,X_n$ on $U\subset M$, 
 such that $X_i(x)=e_i$ for $i=1,\ldots,n$. Then, $\nabla_{X_i}X_j=0$ and, since $\nabla$ is torsion-free, one has that
$$
0=T(X_i,X_j)=\nabla_{X_i}X_j-\nabla_{X_j}X_i-[X_i,X_j]=-[X_i,X_j]\,,\qquad \forall i,j=1,\ldots,n\,.
$$
Hence, there exist coordinates $\{x^1,\ldots,x^n\}$ on a neighbourhood of $x$ such that $X_i=\partial/\partial x^i$ for $i=1,\ldots,n$. Moreover, the Christoffel symbols of the connection vanish on $U$. 

Using the above result and assuming the local existence of a flat torsion-free connection compatible with a symplectic form, one may prove the  Darboux theorem in a very easy manner. In fact, every symplectic form $\omega$ on $M$ can be put into canonical form at any point $x\in M$ for a certain basis $\{e_1,\dotsc,e_n\}$ of $\cT_xM$, i.e.
$$
\omega_x=\sum_{i=1}^ne^{2i-1}\wedge e^{2i}\,.
$$
Recall that on a neighbourhood of every point $x\in M$, one can define a coordinate system $\{x^1,\ldots,x^n\}$ around $x$ so that there exists vector fields $X_i=\partial/\partial x^i$, with $i=1,\ldots,n$, such that
$$
\nabla_{X^i}X^j=0\,,\qquad \forall i,j=1,\ldots,n\,,
$$
and $X_i(x)=e_i$ for $i=1,\ldots,n$. 
Since $\nabla$ is a compatible symplectic connection for $\omega$, one has that
$$
\nabla_{X^i}\omega(X^j,X^k)=0\,,\quad \forall i,j,k=1,\dotsc,n\,.
$$
Hence, one has
$$
\omega = \sum_{i=1}^{2n}\d x^{2i-1}\wedge \d x^{2i}\,. 
$$
In a similar way, but weakening the conditions in Definition \ref{dfn-symplectic-manifold}, we can introduce the concept of presymplectic manifold. Recall that we assume differential forms to have constant rank.

\begin{definition}\label{dfn:presymplectic-manifold}
    A {\it presymplectic form} on $M$ is a closed two-form $\omega\in\Omega^2(M)$ of constant rank. The pair $(M,\omega)$ is called a {\it presymplectic manifold}.
\end{definition}

Let us construct a  prototypical example of presymplectic manifold.  Let $(M,\omega)$ be a symplectic manifold, 
and let $N$ be a submanifold of $M$. Consider the canonical embedding
denoted by $\jmath_N\colon N \hookrightarrow M$, 
and endow $N$ with the induced two-form 
$\omega_N = \jmath_N^\ast\omega$, which
is closed.
Then,  
$(N, \omega_N)$ is a presymplectic manifold provided the rank of $\omega_N$ is constant. To see that the condition on the rank of $\jmath_N^*\omega$ is necessary, let us consider the counter example given by the canonical two form, $\omega=\d x\wedge \d p_x+\d y\wedge \d p_y$, on  $\cT\mathbb{R}^2$ and the immersed submanifold given by $\jmath_{\cT\mathbb{R}}:(x,p_x)\in \cT\mathbb{R}\mapsto  (x,p_x^2/2,0,p_x)\in \cT\mathbb{R}\times \cT\mathbb{R}\simeq \cT \mathbb{R}^2$. Then, $\jmath_{\cT \mathbb{R}}^*\omega=p_x\d x\wedge \d p_x$, which is not symplectic at the zero section of $\cT \mathbb{R}$.

Before introducing the characteristic distribution associated with a presymplectic manifold, let us fix some terminology about distributions. 
A (generalised) \textit{distribution} on $M$ is a subset $D\subset \T M$ such that $D\cap \T_xM$ is a vector subspace of $\T_xM$ for every $x\in M$. 
A distribution $D$ on $M$ is said to be \textit{smooth} if, for every $x\in M$, there exists a neighbourhood $U_x$ of $x$ and (smooth) vector fields $X_1,\ldots,X_k$ on $U_x$ so that $D_y=\langle X_1(y),\ldots,X_k(y)\rangle$ for every $y\in U_x$. 
A generalised distribution $D$ is \textit{regular} if it is smooth and has constant rank. 
A (generalised) \textit{codistribution} on $M$ is a subset $C\subset\cT M$ such that $C_x = C\cap \cT_xM$ is a vector subspace of $\cT_xM$ for every $x\in M$. 
The smooth and/or regular notions introduced for distributions also apply to codistributions.


\begin{definition}
    Given a presymplectic manifold $(M,\omega)$, its {\it characteristic distribution} is  the distribution
    \begin{equation*}
        \C_\omega = \ker\omega = \{ v\in \T M \mid \omega(v,\cdot) = 0 \}\,.
    \end{equation*}
    A vector field $X\in\X(M)$ belonging to $\C_\omega$, i.e. such that $\inn_X\omega = 0$, is called a {\it characteristic vector field} of $(M,\omega)$.
\end{definition}

Note that $\C_\omega = \ker\flat$. In the case of symplectic manifolds, $\flat$ is a vector bundle isomorphism, and thus $\C_\omega = \{0\}$. Moreover, $\C_\omega$ is a distribution because $\omega$ has constant rank. But the kernel of a general closed two-form does not need to be a smooth generalised distribution. For example, $\omega_P=(x^2+y^2)\d x\wedge\d y$ is a closed two-form on $\R^2$, but it is not presymplectic, as its rank is not constant. Moreover, $\ker \omega_P$ is a generalised distribution $(\ker \omega_P)_{(0,0)} = \T_{(0,0)}\R^2$  while $(\ker\omega_P)_{(x,y)}=0$ for every $(x,y)\in \mathbb{R}^2$ different from $(0,0)$. Indeed, $\ker\omega_P$ is not even a smooth generalised distribution.  

\begin{proposition}
    The characteristic distribution $\C_\omega$ of a presymplectic manifold $(M,\omega)$ is integrable.
\end{proposition}
\begin{proof} The integrability of $\C_\omega$ follows from the closedness of the symplectic form $\omega$, the constancy of his rank, and the Frobenius theorem. 
\end{proof}

If $\omega$ is a presymplectic form on $M$, its characteristic distribution is integrable. Moreover, around every $x\in M$, there exists an open neighbourhood $U$ of $x$ such that the space of integral leaves of $\C_\omega$, let us say $U/\C_\omega$, admits a natural manifold structure and the projection $\pi:U\rightarrow U/\C_\omega$ is a submersion. Let us prove this fact.  Since $\mathcal{C}_\omega$ is integrable, the Frobenius theorem ensures that, for every $x\in M$, there exists a local basis of vector fields, 
$\{\partial/\partial x^1,\ldots,\partial/\partial x^k\}$, spanning $\mathcal{C}_\omega$ on a 
 coordinated neighbourhood $U_x$ of $x$ with coordinates $\{x^1,\ldots,x^n\}$. In a small enough open subset $U$ of $U_x$ containing $x$, one can assume that $x^{1},\ldots,x^n$ take values in an open ball of $\mathbb{R}^{n}$. Then, the space of leaves of $U/\mathcal{C}_\omega$ is a manifold of dimension $n-k$ and the mapping $\pi:U\rightarrow \mathbb{R}^{n-k}$ is an open submersion. We will then say that $\mathcal{C}_\omega$ is {\it simple} on $U$. Since $\omega$ is invariant relative to the elements of its characteristic distribution, and it vanishes on them, there exists a unique two-form $\widetilde\omega$ on $U/\C_\omega$ such that $\pi^*\widetilde\omega = \omega$. 
 In this way, $\widetilde\omega$ is closed and nondegenerate because if $\inn_{X_\C}\widetilde\omega = 0$, then there exists a vector field $X$ on $U$ such that $\pi_*X = X_\C$, and then $\inn_X\omega=0$. Since $\C_\omega = \ker\omega = \ker\T\pi$, then $X_\C = 0$. 

With this in mind, we are ready to state the Darboux theorem for presymplectic forms. 
Note that this theorem can be stated since presymplectic forms are assumed to have constant rank. 
Otherwise, it would be difficult to establish a series of canonical forms for closed two-forms even in the most simple cases, e.g. on $\mathbb{R}^2$.  

\begin{theorem}[Darboux theorem for presymplectic manifolds]\label{thm-darboux-presymplectic}
    Consider a presymplectic manifold $(M,\omega)$. Around every point $x\in M$, there exist local coordinates $\{q^i,p_i,z^j\}$, where $i = 1,\dotsc, r$ and $j = 1,\dotsc,d$, such that
    \begin{equation*}
        \omega = \sum_{i=1}^r\d q^i\wedge\d p_i\,,
    \end{equation*}
    where $2r$ is the rank of $\omega$. In particular, if $r=0$, then $\omega=0$ and $d=\dim M$. If $2r=\dim M$, then $d=0$ and $\{q^i,p_i\}$ give a local coordinate system of $M$. 
\end{theorem}
\begin{proof}
    Consider an open neighbourhood $V$ of $x\in M$ where the integral foliation $\mathcal{F}$ defined by the distribution $\C_\omega = \ker\omega$ is simple. Let $P$ be the manifold of leaves of $\restr{\mathcal{F}}{V}$ and let $\pi \colon V\to P$ be the canonical projection. There exists a symplectic form $\bar\omega$ on~$P$ given by $\omega = \pi^\ast \bar\omega$.
The Darboux theorem for symplectic manifolds ensures that there exists an open coordinate neighbourhood $\bar U\subset P$ of $\pi(x)$ with local coordinates $\{\bar q^1,\dotsc,\bar q^r, \bar p_1,\dotsc,\bar p_r\}$ such that $\displaystyle\bar\omega = \sum_{i=1}^r\d \bar q^i\wedge\d \bar p_i$ on $\bar U$. Define $q^i = \bar q^i\circ\pi$ and $p_i = \bar p_i\circ\pi$ for $i=1,\ldots, r$ and we choose $d=\dim M - 2r$ other  functions $z^1,\ldots,z^d$, functionally independent relative to the previous ones. This gives rise to a local coordinate system $\{q^1,\dotsc,q^r,p_1,\dotsc,p_r,z^1,\dotsc,z^d\}$ around~$x$. This chart satisfies the conditions of the theorem.
\end{proof}

The definition of a presymplectic connection is a straightforward generalisation of Definition \ref{def:symplectic-connection} to the presymplectic realm (see \cite{Vai_00}).

\begin{definition}\label{def:presymplectic-connection}
    A {\it presymplectic connection} relative to a presymplectic manifold $(M,\omega)$ is a connection $\nabla$ on $M$ such that $\nabla\omega = 0$.
\end{definition}

The Darboux theorem for presymplectic forms implies that there exists, locally, a flat presymplectic connection. The other way around, the existence of a flat presymplectic connection for a presymplectic manifold $(M,\omega)$ allows us to prove the Darboux theorem as in the case of symplectic forms. In particular, we have proved the following.

\begin{lemma} 
Every presymplectic manifold $(M,\omega)$ admits locally defined flat presymplectic connections $\nabla$, i.e. $\nabla\omega=0$.
\end{lemma}

At this point, it becomes clear that if a differential form admits a compatible flat torsion-less connection, it must be closed. Hence, no flat torsion-less compatible connection exist for contact forms, locally conformally symplectic forms, and other differential forms that are not closed \cite{Gei_08,Vai_85}. We have stressed the word ``torsion-less" unless every connection in this work is assumed to be so, because in the conclusions of this work will show that removing this condition may lead to deal with no-closed differential forms.


\subsection{Cosymplectic and precosymplectic manifolds}
Let us review the definition of cosymplectic \cite{Alb_89,CLL_92,LR_89} and precosymplectic \cite{CLM_94} manifolds, their corresponding Darboux theorems, and their relations to flat cosymplectic and precosymplectic connections.

\begin{definition}\label{dfn:cosymp-manifold}
A {\it cosymplectic structure} in $M$ is a pair
    $(\omega,\eta)$, where $\omega\in\Omega^2(M)$ and $\eta\in\Omega^1(M)$ are closed differential forms such that $\eta$ does not vanish and $\ker \eta\oplus \ker\omega = \T M$. The triple $(M,\omega,\eta)$ is said to be a {\it cosymplectic manifold}.
\end{definition}

Note that a cosymplectic structure on $M$ implies that $M$ is odd-dimensional. The fact that  $\eta$ is non-vanishing implies that $\langle\eta\rangle \oplus\Ima\omega = \cT M$ and $\dim M = 2n+1$ for $n\geq 0$. Then, $(M,\omega,\eta)$ is a cosymplectic manifold if and only if $\eta\wedge\omega^n$ is a volume form on $M$, where we assume that $\omega^0=1$. In particular, a cosymplectic manifold $(M,\omega,\eta)$ yields a presymplectic manifold $(M,\omega)$.  Note that the case $\dim M=1$ may give rise to a cosymplectic manifold according to our definition \cite{THB_22}. 

The \textit{characteristic distribution} of a cosymplectic manifold $(M,\omega,\eta)$ is the rank one distribution given by $\C_\omega = \ker\omega$, and it is often called the \textit{Reeb distribution}. The following proposition states that $(M,\omega,\eta)$ induces a unique distinguished vector field, called \textit{Reeb vector field}, taking values in $\ker\omega$.

\begin{proposition}
    Given a cosymplectic manifold $(M, \omega,\eta)$, there exists a unique vector field $R\in\X(M)$ that satisfies
    \begin{equation*}
        \inn_R\eta = 1\,,\quad \inn_R\omega = 0\,.
    \end{equation*}
\end{proposition}

A cosymplectic manifold $(M,\omega,\eta)$ induces a $\Cinfty(M)$-module isomorphism $\flat\colon\mathfrak{X}(M)\to\Omega^1(M)$ given by $\flat(X) = \inn_X\omega + (\inn_X\eta)\eta$, whose inverse map is denoted by $\sharp = \flat^{-1}$. Then, the Reeb vector field $R$ reads
\begin{equation*}
    R = \sharp\eta\,.
\end{equation*}

Consider the product manifold $\R\times \cT Q$ and the projections $\pi_1\colon\R\times \cT Q\to \R$ and $\pi_2\colon\R\times\cT Q\to \cT Q$ onto the first and second manifolds in $\R\times \cT Q$. If $t$ is the natural coordinate in $\R$ and $\omega_Q$ is the canonical symplectic form on $\T^*Q$, then the triple
\begin{equation} \label{eq:cosymp-canonical-model}
(\R\times \T^\ast Q,\pi_2^\ast\omega_Q,\pi_1^\ast\d t)    
\end{equation}
is a cosymplectic manifold. 
Let us consider the pull-back of $t$ to $\R\times \cT Q$ via $\pi_1$, and the pull-back of some Darboux coordinates $\{q^i,p_i\}$ for $\omega_Q$ to $\R\times \cT Q$ via $\pi_2$. Let us denote such pull-backs in the same way as the original coordinates to simplify the notation. 
Then, in the coordinates $\{t, q^i, p_i\}$, the Reeb vector field of  $(\R\times \cT Q,\pi_2^\ast\omega_Q,\pi_1^\ast\d t)$ read $\partial/\partial t$. Locally, $\pi_2^*\omega_Q=\d q^i\wedge \d p_i$ and $\pi_1^*\d t = \d t$. 
		
\begin{theorem}[Cosymplectic Darboux theorem \cite{LR_89}]\label{thm-darboux-cosymplectic}
    Given a cosymplectic manifold $(M, \omega, \eta)$, there exists, around each point $x\in M$, 
    local coordinates $\{t, q^i, p_i\}$, 
    where $1\leq i\leq n$, such that
    \begin{equation*}
        \eta = \d t \,,\quad \omega = \d q^i\wedge\d p_i \,.
    \end{equation*}
\end{theorem}
\begin{proof}
    Since $(M,\omega)$ is a presymplectic manifold and $\omega$ has corank one, there exist for any point $x\in M$ a neighbourhood $U$ of $x$ with coordinates $\{s,q^i,p_i\}$, with $i=1,\ldots,n$, so that $\omega=\d q^i\wedge \d p_i$. Consider now a potential function of $\eta$, which exists because $\eta$ is closed, and denote it by $t$. Since $\eta\wedge\omega^n$ is a volume form, $\{t,q^i,p_i\}$ is a coordinate system around $x$ and $\eta = \d t$ and $\omega = \d q^i\wedge \d p_i$.
\end{proof}

The Darboux theorem for cosymplectic manifolds states that every cosymplectic manifold is locally diffeomorphic to the canonical model \eqref{eq:cosymp-canonical-model} (see \cite{CLL_92,LR_89}). In Darboux coordinates, the Reeb vector field $R$ for a cosymplectic manifold $(M,\omega,\eta)$ is written as $R = \parder{}{t}$.

The Darboux theorem for cosymplectic structures implies that there exists, around each point, a flat  connection $\nabla$ such that $\nabla\eta = 0$ and $\nabla\omega = 0$. Indeed, $\nabla$ can be chosen to be the connection with zero Christoffel symbols relative to some Darboux coordinates. This justifies the following definition.
\begin{definition} A {\it cosymplectic connection} relative to $(M,\omega,\eta)$ is a connection on $M$ such that $\nabla\eta = 0$ and $\nabla\omega = 0$. 
\end{definition}

Let us show that the existence of flat cosymplectic connections allows us to prove the Darboux theorem for $(M,\omega,\eta)$. At a point $x\in M$, the fact that $\ker \eta_x\oplus \ker \omega_x=\T_xM$ implies that there exists a basis of $\T_xM$ of the form $\{e_1,\ldots, e_{2n+1}\}$ so that $\eta_x=e^{2n+1}$ and $\omega_x=\sum_{i=1}^{n}e^{2i-1}\wedge e^{2i}$ relative to the dual basis $\{e^1,\ldots,e^{2n+1}\}$ in $\cT_x M$. Due to the fact that $\nabla$ is flat, there exists a family of 
 commuting parallel vector fields $X_1,\ldots,X_{2n+1}$ such that $X_i(x)=e_i$ for $i=1,\ldots,2n+1$. Since
$$
\nabla_{X_i}[\eta(X_j)]=0\,,\qquad \nabla_{X_i}[\omega(X_j,X_k)]=0\,,\qquad i,j,k=1,\ldots, 2n+1\,,
$$
the dual basis of differential one-forms $\tau^1,\ldots,\tau^{2n+1}$ to $X_1,\ldots,X_{2n+1}$ is such that
$$
\eta=\tau^{2n+1},\qquad \omega=\sum_{i=1}^n\tau^{2i-1}\wedge \tau^{2i}\,.
$$
Since $X_1,\ldots,X_{2n+1}$ admit a coordinate system so that $X_i=\partial/\partial x_i$, with $i=1,\ldots,2n+1$, then $\tau^{i}=\d x^{i}$ for $i=1,\ldots,2n+1$, and the Darboux theorem for cosymplectic manifolds follows. Note that this is due to the fact that the connection is assumed to be torsion-free.

Cosymplectic manifolds can be generalised by assuming that $\eta\in\Omega^1(M)$ and $\omega\in\Omega^2(M)$ are closed forms on $M$, but $\ker\eta\cap\ker\omega$ is a distribution of fixed rank that is not necessarily zero. This implies that $\omega$ is a presymplectic form on $M$. This gives rise to the definition of a precosymplectic manifold. When $\ker \eta\cap \ker \omega=\{ 0\}$, one retrieves the definition of a cosymplectic manifold. 
\begin{definition}
  A {\it precosymplectic structure} in $M$ is a pair $(\omega,\eta)$, where $\omega\in\Omega^2(M)$ and  $\eta\in\Omega^1(M)$ are closed differential forms such that $\ker \eta\cap\ker \omega$ is a regular distribution strictly included in $\ker \omega$ at every $x\in M$. 
 If $\rk\omega = 2r<\dim M$, the triple $(M,\omega,\eta)$ is said to be a {\it precosymplectic manifold} of rank $2r$.
\end{definition}

It is worth stressing that the fact that $\ker \eta\cap \ker \omega$ is a regular distribution strictly contained in $\ker\omega$ implies that 
    $\eta\wedge\omega^r$ is a non-vanishing form and $\omega^{r+1} = 0$ for a certain fixed $r$, and conversely. Therefore, $\omega$ has constant rank $2r$, with $2r <  \dim M$.
\begin{remark}\label{rmrk:R-times-presymplectic}
    Let $(P,\omega)$ be a presymplectic manifold with Darboux coordinates $\{q^i,p_i,z^j\}$. Consider the manifold $\R\times P$ with the induced coordinates $\{t,q^i,p_i,z^j\}$ obtained as usual, namely, $q^i,p_i,z^j$ are the pull-back to $\mathbb{R}\times P$ of the chosen variables in $P$. Then, $(\R\times P,\pi_2^*\omega ,\pi^*_1\d t)$ is a precosymplectic manifold. In the obtained local coordinates, $\pi^*_2\omega=\d q^i\wedge \d p_i$ while $\pi^*_1\d t$  is denoted by $\d t$ to simplify the notation.
\end{remark}

\begin{remark}
    Consider the regular distribution $D=\ker\omega\cap\ker\eta$ of a precosymplectic manifold $(M,\omega,\eta)$. Then, $D$ is involutive because $\ker\omega$ and $\ker \eta$ are so. The foliation associated with $D$ defines a local projection
    \begin{equation}\label{precosymplectic-fibration}
        \pi\colon M\to \widetilde M = M / (\ker\omega\cap\ker\eta)\,,
    \end{equation}
where $\widetilde{M}$ is the quotient manifold of the leaves of $D$. Recall that we are assuming that $\widetilde{M}$ is a manifold for simplicity. Indeed, one of the general assumptions of our paper is that manifold structures and other existing mathematical local structures are defined globally. In reality, one can only ensure that for every $x\in M$ and a local neighbourhood $U_x$ of $x$, the space $M/(\ker\omega\cap\ker\eta)$ is a manifold. Hence, by our general assumptions, there exists a unique cosymplectic structure $(\widetilde\omega,\widetilde\eta)$ on $\widetilde M$ such that $\pi^\ast\widetilde\omega = \omega$ and $\pi^\ast\widetilde\eta = \eta$. 
\end{remark}

As in the case of cosymplectic manifolds, we can define special types of vector fields for precosymplectic manifolds.

\begin{definition}\label{dfn-precosymplectic-reeb}
    Given a precosymplectic manifold $(M,\omega,\eta)$, a vector field $X\in\mathfrak{X}(M)$ satisfying
    \begin{equation*}
        \inn_X\omega = 0\,,\quad \inn_X\eta = 1\,,
    \end{equation*}
    is called a {\it Reeb vector field}. The space generated by Reeb vector fields, namely $\ker \omega$, is called the {\it Reeb distribution} of $(M,\omega,\eta)$.
\end{definition}

Note that, if $R\in\X(M)$ is a Reeb vector field, then $R+Y$ is also a Reeb vector field for every $Y \in\ker\omega\cap\ker\eta$. In other words, Reeb vector fields for precosymplectic manifolds need not be univocally defined.



Finally, let us state the Darboux theorem for precosymplectic manifolds, whose proof seems, as far as we know, to be absent in the literature. Nevertheless, it is always implicitly assumed that it holds \cite{CLM_94,LR_89} and it is quite straightforward. 
	
\begin{theorem}[Darboux Theorem for precosymplectic manifolds]\label{thm-darboux-precosymplectic}
    Let $(M,\omega,\eta)$ be a precosymplectic manifold with $\rk\omega=2r\leq \dim M-1$. 
    For every $x\in M$, there exist local coordinates $\{t, q^i, p_i, z^j\}$ around $x$, where $1\leq i\leq r$ and $1\leq j\leq \dim M - 1 - 2r$, such that
    \begin{equation}\label{Eq:DarbouxPreCo}
        \eta = \d t\,,\quad \omega = \sum_{i=1}^r\d q^i\wedge\d p_i \,.
    \end{equation}
\end{theorem}
\begin{proof}
    Since $\omega$ is a presymplectic form, there exist coordinates $\{q^i,p_i,z'_i\}$ on a neighbourhood $U$ of $x$ such that $\omega =\sum_{i=1}^r \d q^i\wedge \d p_i$. Since $(\ker\eta\cap \ker\omega )^
    \circ=\langle \eta\rangle \oplus {\rm Im}\,\omega$ and $\eta$ does not vanish, one has that $\rank \omega\leq \dim M-1$. One the other hand, $\eta$ is closed and, therefore, there exists a function $t$ on $U$, where $U$ can be chosen smaller if necessary, such that $\eta = \d t$ and $\omega=\d q^i\wedge \d p_i$. Since $\eta\wedge \omega^r$ does not vanish, $\{t,q^i,p_i\}$ are functionally independent functions. Finally, one can choose additional functionally independent coordinates, $z^1,\dotsc,z^n$ with respect to $\{t,q^i,p_i\}$ and  \eqref{Eq:DarbouxPreCo} will hold.
\end{proof}

As in the previous cases, there exists a  locally defined flat connection $\nabla$ whose Christoffel symbols vanish on the chosen Darboux coordinates. Then, 
 $\eta$ and $\omega$ become parallel differential forms relative to $\nabla$. This motivates the following natural definition.

\begin{definition} A {\it precosymplectic connection}  relative to a precosymplectic manifold $(M,\eta,\omega)$ is a connection on $M$ such that $\nabla\eta=0$ and $\nabla\omega=0$.
\end{definition}

Note that, as previously, the existence of a flat precosymplectic connection allows one to provide a brief proof of the Darboux theorem for precosymplectic manifolds.

\section{\texorpdfstring{$k$}--symplectic and \texorpdfstring{$k$}--presymplectic manifolds}\label{sec3:k-symp}

Let us introduce and provide Darboux theorems for $k$-symplectic manifolds. This will give a new, complementary approach, to the classical results \cite{Awa_92,CB_08} and some new more modern approaches \cite{FG_13}. Moreover, we will discuss the existence of Darboux theorems for $k$-presymplectic manifolds. 
Furthermore, this will be done by providing new simpler, shorter and more geometrical proofs of Darboux theorems for $k$-symplectic manifolds while giving more details and, as far as we know, a new Darboux theorem for linear spaces \cite{Awa_92}. Additionally, we will give a new proof about the existence of a complementary for a polarisation that is isotropic relative to the differential two-forms of a $k$-symplectic structure.   
On the other hand, Darboux theorems give rise to the hereafter called {\it flat $k$-symplectic and $k$-presymplectic connections}, 
which, in turn, lead to other proofs of respective Darboux theorems. It is worth noting that an alternative, somehow different, development of these ideas for the $k$-symplectic case can be found in \cite{CB_08}. Moreover, some new structures will arise in our approach and our results concerning $k$-presymplectic manifolds seem to be absolutely new.

\begin{definition}
\label{dfn:k-symplectic-manifold}
Let $M$ be an $n(k{+}1)$-dimensional manifold. 
A {\it$k$-symplectic structure} on $M$ is a family 
$(\omega^1,\dotsc, \omega^k,V)$, 
where $V$ is an  integrable distribution on~$M$ of rank $nk$, and
$\omega^1,\dotsc,\omega^k$ are closed differential $2$-forms on~$M$ satisfying that
\begin{enumerate}[(1)]
\item 
$\restr{\omega^\alpha}{V\times V}=0$,
for $1\leq \alpha\leq k$,
\item 
$\bigcap_{\alpha=1}^k \ker \omega^\alpha = \{0\}$.
\end{enumerate}
Under the above hypotheses, $(M,\omega^1,\dotsc,\omega^k,V)$ is called a {\it$k$-symplectic manifold}. We call $V$ a {\it polarisation} of the $k$-symplectic manifold. 
\end{definition}

Our notion of $k$-symplectic manifold matches the one given by A.~Awane \cite{Awa_92, AG_00}. 
Moreover, it is equivalent to the concepts of 
\textit{standard polysymplectic structure} of C.~Günther \cite{Gun_87} 
and 
\textit{integrable $p$-almost cotangent structure} introduced by M.~de León \textit{et al} \cite{LMS_88, LMS_88a}. 
In the case $k=1$, Awane's definition reduces to the notion of \textit{polarised symplectic manifold}, 
that is a symplectic manifold with a Lagrangian foliation. 
We will illustrate in forthcoming examples that the distribution~$V$ is needed to ensure the existence of a particular type of Darboux coordinates.

In fact, G\"unther calls polysymplectic manifolds the differential geometric structures obtained from our definition by removing the existence of the distribution $V$. 
Meanwhile, a \emph{standard} polysymplectic manifold in G\"unther's paper is a polysymplectic manifold admitting an atlas of Darboux coordinates. Note that a polysymplectic manifold may have atlas of Darboux coordinates without a distribution $V$. In a particular case, if we think of symplectic manifolds as a one-symplectic manifold, then it is clear that it has local Darboux coordinates, but the standard symplectic manifold on the sphere does not have a polarisation \cite{Awa_20}.  Then, G\"unther's definition is more general than ours, while it is equivalent to our definition if the compatibility of two charts Darboux coordinates $\{y^i,p^\alpha_i\}$ and $\{x^i,\pi^\alpha_i\}$ involves that $x=x(y)$ and the momenta are transformed accordingly, namely $\pi^\alpha=\pi^\alpha(y,p)$ are the momenta of the $\{x^i\}$. Otherwise, the equivalence is only local.

Let us provide a Darboux theorem at the tangent space of a point of a $k$-symplectic manifold. Since every $k$-symplectic manifold $(M,\omega^1,\dotsc,\omega^k,V)$ induces at every $\T_xM$ for $x\in M$ a so-called $k$-{\it symplectic vector space}, Theorem \ref{thm:k-symp-puntual} can be understood as a Darboux theorem for $k$-symplectic vector spaces.

\begin{theorem} {\bf ($k$-symplectic linear Darboux theorem)} \label{thm:k-symp-puntual}
Assume that $(M,\omega^1,\ldots,\omega^k,V)$ 
is a $k$-symplectic manifold.
For every $x\in M$, 
there exists a basis 
$\{e^1,\ldots,e^n;e_1^\beta,\ldots,e_n^\beta\}_{\beta=1,\ldots,k}$  of $\cT_xM$ such that 
$$
\omega^\beta=\sum_{i=1}^ne^i\wedge e_i^\beta\,,\qquad V=\bigoplus_{\alpha=1}^kV_\alpha,\qquad  V_\beta=\langle e^1_\beta,\ldots, e^n_\beta\rangle,\qquad \beta=1,\ldots,k\,.
$$
Note that $\{e_1,\ldots,e_n,e_\beta^1,\ldots,e_\beta^n\}$ is the dual basis in $\T_xM$.
\end{theorem}
\begin{proof} The result amounts to the Darboux theorem for symplectic linear spaces for $k=1$. Hence, let us assume $k>1$. Since $\{0\}=\bigcap_{\alpha=1}^k\ker \omega^\alpha$, 
 one has that
\begin{equation}
\label{con}
\T^*_xM = \Ima \omega^1_x + \dotsb + \Ima \omega^k_x \,,\qquad \forall x\in M\,.
\end{equation}
Although all posterior structures in this proof refer to the point $x$, the point  will be omitted to simplify the notation. Since $\omega^\beta|_{V\times V}=0$, one has that $\omega^\beta(V)\subset V^\circ$ for $\beta=1,\ldots,k$. If $W$ is a regular distribution supplementary to $V$, then
$\rk W=n$, and $\rk \omega ^\beta\,(W )\leq n$. Note that
$$
\omega^1(V)+\dotsb+\omega^k(V)\subset V^\circ\,.
$$
Due to \eqref{con} and the above discussion, one has that 
$$
\omega^1(W)+\dotsb+\omega^k(W)
$$
is a distribution of rank $nk$, at least. 
This implies that $\rk\omega^\beta(W)=n$ and
$$ \omega^1(W)\oplus\dotsb\oplus\omega^k(W)\oplus V^\circ = \T^*M\,,\qquad V^\circ=\omega^1(V)+\ldots+\omega^k(V)\,. $$
If $\omega^\alpha(v+w)=0$, where $v\in V$ and $w\in W$, then $\omega^\alpha(v)=-\omega^\alpha(w)$. Since $\omega^\alpha(W)\cap \omega^\alpha(V)=0$ and $\rk \omega^\alpha|_W=n$, then $\omega^\alpha(v)=0$ and $\omega^\alpha(w)=0$, which implies that $w=0$ and $v\in \ker \omega^\alpha$. Hence, $\ker \omega^\alpha\subset V$. 
We can consider the distributions 
$$
V_\beta = 
 \bigcap_{\substack{1 \leq \alpha \leq k \\ \alpha \neq \beta}} \ker \omega^\alpha\,, \qquad \beta=1,\ldots,k\neq 1\,,\qquad \text{or}\qquad V^1=V\ \ (k=1)\,.
$$
Note that $\omega^\beta(V_\alpha)=0$ for $\alpha\neq \beta$ and for every $\alpha,\beta=1,\ldots,k$. 

Let $\{w_1,\ldots, w_n\}$ be a basis of $W$. Since $\omega^\alpha(W)$ has rank $n$ and its elements do not belong to $V^\circ$, then the restrictions of $\omega^\alpha(w_1),\ldots,\omega^\alpha(w_n)$ to $V$ are linearly independent and there exist $v_1,\ldots,v_n$ in $V$ such that $\omega^\alpha(v_1),\ldots,\omega^\alpha(v_n)$ are linearly independent on $W$, e.g. $\omega^\alpha(w_i,v_j) = \delta_{ij}$ for $i,j=1,\ldots,n$. Hence, $\rk\omega^\alpha(V)\geq n$. Since $\omega^\alpha(V)\subset V^\circ$, then $\rk \omega^\alpha(V)=n$. In particular, $\rk \omega^\alpha(V)=n$ for every $\alpha=1,\ldots,k$. Since $\bigcap_{\alpha=1}^k\ker\omega^\alpha=0$ and ${\rm Im}\,\omega^\alpha(V)\subset V^\alpha$ for $\alpha=1,\ldots,k$, it follows that $\phi:v\in V\mapsto (\omega^1(v),\ldots,\omega^k(v))\in V^\circ\oplus\overset{(k)}{\dotsb}\oplus V^\circ$, where $\oplus$ stands for a Whitney sum of vector bundles in the natural way, is injective. Hence, $\phi$ becomes an isomorphism and  $V\simeq \bigoplus_{\alpha=1}^kV_\alpha$. Indeed, $v=\sum_{\alpha=1}^k\phi^{-1}({\rm pr}_{\alpha}(v))$, where ${\rm pr}_\alpha:(w_1,\ldots,w_k)\in V^\circ \oplus\ldots\oplus V^\circ\mapsto (0,\ldots,0,w_\alpha,0,\ldots,0)\mapsto V^\circ \oplus\ldots\oplus V^\circ$, is the corresponding decomposition. 

Since $V = \bigoplus_{\alpha=1}^k V_\alpha$ and $\omega^\beta(V_\beta)\subset V^\circ$ has the same rank as $V^\beta$, it follows that $\rk V_\beta=n$. Hence, one can consider a basis $\{e^1,\ldots,e^n\}$ of $V^\circ$. There exists a basis $f_\beta^1,\ldots,f_\beta^n$ of each $V^\beta$ such that $\omega^\alpha(f_\beta^i)=-e^i\delta_\alpha^\beta$ for $i=1,\ldots,n$ and $\alpha,\beta=1\,\ldots,k$.
Considering a dual basis $\{f^\beta_i,e^i\}$ of $\T_x^*M$, one has that
\begin{equation}\label{eq:Form1}
\omega^\beta = e^i\wedge f^\beta_i + c_{ij}^\beta e^i\wedge e^j\,,
\qquad \beta=1,\ldots,k\,.
\end{equation}
If $e^\beta_i=f^\beta_i+c_{ij}^\beta e^j$, then
$$
\omega^\beta = \sum_{i=2}^n e^i\wedge e^\beta_i\,,
\qquad \beta=1,\ldots,k\,.
$$
Note that the change on the covectors $e^\beta_i$ implies that, in the dual bases to the bases $\{e^i,e_i^\alpha\}$ and $\{e^i,f_i^\alpha\}$ in $\T_x^*M$, one has that $f_\alpha^i=e_\alpha^i$  for $\alpha=1,\ldots,k$ and $i=1,\ldots,n$. Hence,
$
V_\alpha=
\langle f_\alpha^1,\ldots,f^n_\alpha\rangle=\langle e_\alpha^1,
\ldots,e_\alpha^n\rangle$ for $\alpha=1,\ldots,k$. \end{proof}

It stems from Theorem \ref{thm:k-symp-puntual} that $\omega^1,\ldots,\omega^k$ have constant rank. This fact comes from the definition of $k$-symplectic structure, the dimension of $M$, and the rank of $V$. Note also that the last paragraph in the proof of Theorem \ref{thm:k-symp-puntual} can be almost straightforwardly changed to put a symplectic linear form and a Lagrangian subspace into a canonical form.

\begin{definition} Given a $k$-symplectic manifold $(M,\omega^1,\ldots,\omega^k,V)$ with $k\neq 1$, we set
$$
V_\beta= \bigcap_{\substack{\alpha=1 \\ \alpha\neq \beta}}^k\ker\omega^\alpha\,,\qquad \beta=1,\ldots,k.
$$
\end{definition}
\begin{lemma}\label{lem:first-integrals}
For a $k$-symplectic manifold $(M,\omega^1,\ldots,\omega^k, V)$, the distributions $V^1,\ldots,V^k$ satisfy that every $x\in M$ admits a coordinate system 
$\{y^1,\ldots,y^n;y^\alpha_1,\dotsc,y_n^\alpha\}$ on a neighbourhood of $x$ so that
$$
V_\alpha =
\left\langle \frac{\partial}{\partial y^\alpha_1},\ldots,\frac{\partial}{\partial y^\alpha_n} \right\rangle ,\qquad \alpha=1,\ldots,k.
$$
\end{lemma}
\begin{proof} Let $y^1,\ldots,y^n$ be common functionally independent first-integrals for all vector fields taking values in $V$. If $k=1$, the result follows trivially, so we assume $k> 1$. Given different $\alpha_1,\ldots,\alpha_{k-1}\in\{1,\ldots,k\}$, one has that 
$$
V_{\alpha_1}\oplus\dotsb\oplus V_{\alpha_{k-1}}=\ker \omega^\beta,
$$
where $\beta$ is the only number in $\{1,\ldots,k\}$ not included in $\{\alpha_1,\ldots,\alpha_{k-1}\}$. 
Hence, the distribution $V_{\alpha_1}\oplus\ldots\oplus V_{\alpha_{k-1}}$ has rank $n(k-1)$, it is integrable because $\omega^\beta$ is closed, and the vector fields taking values in it have $n$ common local first-integrals $y^\beta_1,\ldots,y^\beta_n$ such that $\d y^\beta_1\wedge \ldots \wedge \d y^\beta_n\wedge \d y^1\wedge \ldots\wedge \d y^n\neq 0$. 
By construction,  $\{y^1_1,\ldots,y^1_n,\ldots,y^k_1,\ldots,y^k_n,y^1,\ldots,y^n\}$ becomes a local coordinate system on $M$ and 
$$
V_\alpha=\left(\bigcap_{i=1}^n\ker\d y^i\right)\cap\left(\bigcap_{\substack{\beta\neq \alpha\\i=1,\ldots,n}}\ker \d y^\beta_i\,\right).
$$
Moreover, $\dfrac{\partial}{\partial y^\beta_1},\ldots,\dfrac{\partial}{\partial y^\beta_n}$ vanish on all coordinates $y^{\alpha}_1,
\ldots,y^\alpha_n$ with $\alpha\neq \beta$. 
Hence, 
$$
V_\beta=\left\langle\frac{\partial}{\partial y^\beta_1},\ldots,\frac{\partial}{\partial y^\beta_n}\right\rangle=\bigcap _{\alpha\neq \beta=1}^k\ker \omega^\alpha,\qquad \beta=1,\ldots,k.
$$
\end{proof}

\begin{theorem}[{\bf Darboux theorem for $k$-symplectic manifolds}]
\label{thm:Darboux k-simp}
Let $(M,\omega^1,\dotsc, \omega^k,V)$ be a $k$-symplectic manifold. 
Around every point $x\in M$, 
there exist local coordinates $\{q^i,p^\alpha_i\}$, 
with $1\leq i\leq n$ and $1\leq\alpha\leq k$, 
such that
\begin{equation}\label{Eq:DarKSym}
\omega^\alpha =\sum_{i=1}^n \d q^i\wedge\d p^\alpha_i\,,
\quad 
V = \left\langle\parder{}{p^\alpha_i}\right\rangle_{\tiny \begin{gathered}i=1,\ldots,n,\\\alpha=1,\ldots,k\end{gathered}} . 
\end{equation}
\end{theorem}
\begin{proof} 
By our  Darboux theorem for $k$-symplectic vector spaces, namely Theorem \ref{thm:k-symp-puntual}, 
there exists a basis 
$\{e^1,\dotsc,e^n;e_1^\alpha,\dotsc,e_n^\alpha\}_{\alpha = 1,\dotsc,k}$ 
of $\cT_xM$ 
such that  $\omega^\alpha_x =\sum_{i=1}^n e^i\wedge e_i^\alpha$ for $\alpha = 1,\dotsc,k$. 
The basis is chosen so that the dual basis  $\{e_1,\dotsc,e_n,e^1_\alpha,\dotsc,e^n_\alpha\}$ , with $\alpha=1,\ldots,k$, 
is such that  $V=\langle e_\alpha^i\rangle _{\substack{\alpha=1,\dotsc,k \\ i=1,\dotsc,n}}$.  Recall that the subspaces in $\T_xM$ of the form
$$
V_{\beta x} = 
\bigcap_{\substack{\alpha=1 \\ \alpha\neq \beta}}^k\ker \omega_x^\alpha = 
\left\langle e_1^\beta,\dotsc,e_n^\beta\right\rangle ,
\quad \beta = 1,\dotsc,k\,,
$$
satisfy that $V_x = \bigoplus_{\beta = 1}^kV_{\beta x}$. 
 By Lemma \ref{lem:first-integrals}, there exist variables 
$\{y^j,y_j^\beta\}$, with $j = 1,\dotsc,n$ and $\beta = 1,\dotsc,k$, 
such that, locally, 
$\displaystyle V^\beta = \Big\langle \parder{}{y^\beta_1},\dotsc, \parder{}{y^\beta_n}\Big\rangle$, 
with $\beta = 1,\dotsc,k$. 
Moreover, $\ker \omega^\beta = \bigoplus_{\alpha\neq \beta}V^\alpha$. 
Using previous results and since $\restr{\omega^\beta}{V\times V} = 0$, 
we have 
$\omega^\beta = f_i^{j\beta} \d y^i\wedge \d y^\beta_j + g_{ij}\d y^i\wedge \d y^j$ 
for certain functions $g_{ij},f_i^{j\beta}$, with $i,j=1,\dotsc,n$ and $\beta = 1,\dotsc,k$. 
Since $\d\omega^\beta = 0$, it follows that 
$f^{j\beta}_i = 
f^{j\beta}_i(y^l,y_l^\beta)$  and $g_{ij}=g_{ij}(y^l,y_l^\beta)$ for $i,j,l=1,\ldots,n$. Therefore, each $\omega^\beta$ can be considered as a differential two-form on $\R^{2n}$. 
Moreover, each $V^\beta$ can be then considered as a Lagrangian distribution of a symplectic two-form $\omega^\beta$, 
when it is considered as a differential two-form on $\R^{2n}$. Consequently,  for a fixed $\beta$, one has
$$
0 = -\parder{}{y^\beta_j} y^i = 
\inn_{X^\beta_{y^i}}\inn_{\partial/\partial y^\beta_j} \omega^\beta,
\quad \quad i,j=1,\dotsc,n\quad\Longrightarrow\quad 
X^\beta_{y^i}\in (V_\beta)^\perp=V_\beta\,,
\quad i=1,\dotsc,n\,.
$$
Note that the orthogonal is relative to the restriction of $\omega^\beta$ to $\mathbb{R}^{2n}$. Whatever, by additionally considering $\iota_{X_{y^i}^\beta} \omega^\alpha=0$ for $\alpha\neq \beta$, one can also see that  $X^\beta_{y^i}$ becomes a vector field taking values in $V^\beta$.
 What follows is an adaptation of the Liouville--Mineur--Arnold theorem (see also \cite{CM_18}).  
Since $V^\alpha$ is integrable, we can consider a leaf $F$ of $V^\alpha$ and its canonical inclusion $\jmath_F:F\hookrightarrow M$. 
Let us define the map $\zeta \colon x\in M\mapsto (y^1(x),\ldots,y^n(x))\in \R^n$. 
Consider a regular point $x'\in M$ of $\zeta$. 
Since the map $\zeta$ is regular in an open neighbourhood of $x'$, there exist vector fields $Y_1,\dotsc,Y_n$ on a neighbourhood of $x'$ such that $Y_i$ 
  and $\displaystyle \parder{}{y^i}$ on $\mathbb{R}^n$ are $\zeta$-related for $i=1,\ldots,n$. 
 Consider the inner contractions $\Theta_i^\alpha=\inn_{Y_i}\omega^\alpha$ for $i=1,\ldots,n$ on a neighbourhood of $x'$ in $M$ and the vector fields $X^\alpha_{y^i}$, which take values in $V_\alpha$. 
Then,
$$ 
\inn_{X^\alpha_{y^i}}{\Theta_j^\alpha}=
\inn_{X^\alpha_{y^i}}\inn_{Y_j}\omega^\alpha=\omega^\alpha(Y_j,X^\alpha_{y^i})=
-\omega^\alpha(X^\alpha_{y^i},Y_j)=
-Y_jy^i=
-\delta^i_j,\qquad i,j=1,\ldots,n.
$$
Hence, given two vector fields $X^\alpha_{y^i},X^\alpha_{y^j}$, one has
$$
(\d\Theta_\ell^\alpha)(X^\alpha_{y^i},X^\alpha_{y^j})=
X^\alpha_{y^i}\Theta^\alpha_\ell(X^\alpha_{y^j})-X^\alpha_{y^j}\Theta^\alpha_\ell(X^\alpha_{y^i})-\Theta^\alpha_\ell([X^\alpha_{y^i},X^\alpha_{y^j}])=
0\,.
$$
The latter is due to the fact that $[X^\alpha_{y^i},X^\alpha_{y^j}]$ is the Hamiltonian vector field of $\{y^i,y^j\}=X^\alpha_{y^j}y^i=0$ because $X_{y^j}$ takes values in  $V_\alpha$. 
Thus, $\jmath_F^*\Theta^\alpha$ is closed and there exists a potential $\jmath_F^*\Theta_i^\alpha=\d p^\alpha_i$. 
And recalling that $\omega^
\alpha|_{V_\alpha\times V_\alpha}=0$, it follows that $\omega^\alpha=\d y^i\wedge \d p^\alpha_i$. Moreover, it follows that
$$
V_\alpha=\left\langle \frac{\partial}{\partial p^\alpha_i}\right\rangle,\qquad \alpha=1,\ldots,k,
$$
and $V$ takes the proposed form.
\end{proof}

Let us recall that the above proof could have been cut by the half by referring straightforwardly to the Liouville--Mineur--Arnold theorem, as since $\{y_i,y_j\}=0$, with $i,j=1,\ldots,n$, that theorem implies that there are functions $p^\beta_1,\dotsc, p^\beta_n$ with $\beta=1,\ldots,k$ such that $\omega^\beta=\d y^i\wedge \d p^\beta_i$ for each $\beta=1,\ldots,k$. Instead, we decided to give a complete, self-contained proof. Without this full explanation, Theorem \ref{thm:Darboux k-simp} would probably be the shortest direct proof of the Darboux theorem for $k$-symplectic manifolds in the literature. Although Theorem \ref{thm:Darboux k-simp} relies on Lemma \ref{lem:first-integrals} and the $k$-symplectic linear Darboux  theorem, Lemma \ref{lem:first-integrals} is a rather straightforward geometric result, which was described carefully to verify all the details, and only the fact that $V=\bigoplus_{\alpha=1}^kV^\alpha$ is needed from the $k$-symplectic linear Darboux theorem to prove our full $k$-symplectic Darboux theorem.  

Moreover, note that one could have assumed that Darboux coordinates only are concerned with the canonical expressions of $\omega^1,
\ldots,
\omega^k$. It turns out that given the conditions on the distribution, once we put $\omega^1,
\ldots,
\omega^k$ in a canonical manner, we also put a basis of $V$ in the desired form. We will see in next section that this is not the case for Darboux coordinates for other structures.

\begin{definition} Given a $k$-symplectic manifold $(M,\omega^1,\ldots,\omega^k,V)$, we call {\it $k$-symplectic Darboux coordinates} the coordinates allowing us to write $\omega^1,\ldots,\omega^k$ and $V$ in the form \eqref{Eq:DarKSym}.
\end{definition}

The $k$-symplectic Darboux coordinates will be called just Darboux coordinates when it does not lead to any misunderstanding. Note that the proof of Theorem~\ref{thm:Darboux k-simp} shows that $k$-symplectic Darboux coordinates induce the existence of a distribution $V'=\langle \partial/\partial y^1,\ldots,\partial/\partial y^n\rangle$ that allows us to state the following result.
        
\begin{corollary} Every $k$-symplectic manifold $(M,\omega^1,\ldots,\omega^k,V)$ admits, locally, a supplementary integrable distribution $V'$ on $M$ such that $V\oplus V'=\T M$ and $\omega^\alpha |_{V'\times V'}=0$ for $\alpha=1,\ldots,k$.
\end{corollary}

The canonical model of a $k$-symplectic manifold is the cotangent bundle of $k^1$-covelocities, namely $\bigoplus^k\cT Q = \cT Q \oplus \overset{k}{\dotsb} \oplus \cT Q$  (the Whitney sum of $k$ copies of the cotangent bundle of a manifold $Q$),  equipped with the distribution $V = \ker \T\pi$, where $\pi^\alpha\colon \bigoplus^k\cT Q \to \cT Q$ and $\pi:\bigoplus^k\cT Q \to  Q$ are the canonical projections onto the $\alpha$-th component and $Q$ respectively, and the canonical presymplectic two-forms $\omega^\alpha = (\pi^\alpha)^*\omega$ with $\alpha=1,\ldots,k$, where $\omega$ stands for the canonical symplectic two-form in $\cT Q$. 
In this model, natural coordinates are Darboux coordinates, 
and the $k$-symplectic Darboux theorem states that $k$-symplectic manifolds are locally diffeomorphic to a cotangent bundle of $k^1$-covelocities. Meanwhile, the distribution $V'$ is a distribution in $\bigoplus_{\alpha=1}^k\cT Q$  whose leaves project diffeomorphically onto $Q$.

As in the previous sections, one can introduce the notion of compatible connection with a $k$-symplectic manifold \cite{Bla_09,CB_08}.

\begin{definition}
A {\it $k$-symplectic connection} on a $k$-symplectic manifold $(M,\omega^\alpha,V)$ is a  connection $\nabla$ on $M$ such that $\nabla\omega^\alpha = 0$ for every $\alpha = 1,\dotsc,k$.
\end{definition}

Again, Darboux coordinates allow us to define, locally, a connection, $\nabla$, such that $\nabla \omega^\alpha=0$ for $\alpha=1,\ldots,k$. And vice versa, the $k$-symplectic linear Darboux Theorem allows us to put $\omega^1,\ldots,\omega^k$ and the distribution $V$ into a canonical form on the tangent space at a point and, a flat connection compatible with the $k$-symplectic manifold enables us to expand this canonical form to an open neighbourhood of the initial point where $\omega^1,\ldots,\omega^k$ and $V$ take the form \eqref{Eq:DarKSym}.

It is worth recalling the interesting work \cite{CB_08}, where connections compatible with $k$-symplectic structures are studied. These connections depend on the existence of certain foliations and are canonical once such foliations are given. By using such foliations and distributions, the Darboux theorem can be proved. 
We find that our approach here is more direct than that in \cite{CB_08} and the Darboux theorem is given in our work more geometrically. 
Note that a $k$-symplectic Darboux theorem also appears as a particular case of the multisymplectic theory in \cite{FG_13}.   

Now, let us study Darboux theorems for $k$-presymplectic manifolds (see \cite{FY_13,GMR_09} for some previous results on this case). This case poses several fundamental problems.  First, there exist several possible definitions of $k$-presymplectic manifolds depending on their possible applications or representative cases. Some possible definition of $k$-presymplectic manifold can be found in \cite{FY_13}. Meanwhile, \cite{GMR_09} defines a $k$-presymplectic manifold as a manifold equipped with $k$ closed two-forms. It is clear that we will not have Darboux coordinates with such a general definition. As shown next, a direct analogue of the Darboux coordinates is not available in some of the possible definitions of $k$-presymplectic structure, while cases that admit Darboux coordinates may not be of physical interest. Let us give a brief analysis of this matter. 

\begin{definition}\label{Def::presym}
Let $\bigoplus_{\alpha=1}^k \T^*Q$ be endowed with its canonical $k$-symplectic structure $\omega^1,\ldots,\omega^k$ and let $\pi \colon \bigoplus_{\alpha=1}^k \T^*Q \rightarrow Q$ be the canonical projection onto $Q$. 
A {\it canonical foliated $k$-presymplectic manifold} is a tuple  $(S,\omega_S^1,\ldots,\omega_S^k)$ given by a submanifold $S\subset \bigoplus_{\alpha=1}^k \T^*Q$ 
such that $\pi|_S \colon S \rightarrow Q$ 
is a fibre bundle and $S$ is endowed, for $\jmath_S \colon S \rightarrow \bigoplus_{\alpha=1}^k \T^*Q$ being the canonical inclusion,  with the $k$ differential two-forms $\omega^\alpha_S=\jmath^*_S\omega^\alpha$, for $\alpha=1,\ldots,k$. The rank of the fibration $\pi|_S:S\rightarrow Q$ is called the rank of $(S,\omega^1_S,\ldots,\omega^k_S)$ while $\omega^1_S,\ldots,\omega^k_S$ are called a {\it canonical foliated $k$-presymplectic structure}. 
\end{definition}
More generally, the above gives rise to the following definition.

\begin{definition}
A {\it foliated $k$-presymplectic manifold} is a tuple $(M,\omega^1,\ldots,\omega^k)$ such that there exists a
canonical foliated $k$-presymplectic manifold $(S,\omega^1_S,\ldots,\omega^k_S)$ and a global  diffeomorphism $\phi:M\rightarrow S$  such that $\phi^*\omega_S^\alpha=\omega^\alpha$ for $\alpha=1,\ldots,k$. A foliated $k$-presymplectic manifold  $(M,\omega^1,\ldots,\omega^k)$ is {\it exact} if $\omega^1,\ldots,\omega^k$ are exact. 
\end{definition}

It is worth noting that the previous definition also makes  sense for $\phi$ being, only, a local diffeomorphism. In that case, the main results to be displayed afterward remain valid, but many more technical details are to be considered to prove them. To keep our presentation simple and highlight the main ideas about Darboux coordinates, which are generically local, we have defined $\phi$ to be a global diffeomorphism.

Definition \ref{Def::presym} implies that $\omega_S^1,\ldots,\omega_S^k$ admit a natural distribution $V=\ker \T\pi\cap \T S$ of rank $\dim S-\dim Q$ such that $\omega_S^\alpha|_{V\times V}=0$ for $\alpha=1,\ldots,k$. If $S=\bigoplus_{\alpha=1}^k\cT Q$, then $V=\ker \T\pi$ and $S$ gives rise to  a $k$-symplectic structure admitting Darboux coordinates.

Let us illustrate by means of a simple example why a Darboux $k$-presymplectic theorem does not exist for general foliated $k$-presymplectic manifolds. It is worth noting that Darboux coordinates for families of closed differential forms are, at the very last instance, a way of writing them in a coordinate system so that their associated coordinates are constant. The following theorem shows that this is impossible for general $k$-presymplectic manifolds.

\begin{theorem} \label{Th::NoGo}
Every rank-zero exact canonical foliated $k$-presymplectic structure is equivalent to $k$ exact differential two-forms on $Q$.
\end{theorem}
\begin{proof}
An exact canonical foliated $k$-presymplectic manifold $(S\subset \bigoplus_{\alpha=1}^k\cT Q,\omega_S^1,\ldots,\omega_S^k)$ 
 gives rise, as $S$ is diffeomorphic to $Q$ via $\pi|_S:S\subset 
 \bigoplus_{\alpha=1}^k\cT Q\rightarrow Q$, to a unique family of exact differential two-forms, $\omega^1_Q,\ldots,\omega^k_Q$, on $Q$ satisfying that $\pi|_S^*\omega_Q^1=\omega_S^1,\ldots,\pi|_S^*\omega_Q^k=\omega_S^k$. 

Conversely, $k$ exact presymplectic two-forms $\omega^1_Q,\ldots, \omega^k_Q$ on~$Q$ with potentials $\theta^1,\ldots,\theta^k$ give rise to a section  
$S = \{(q,\theta^1(q),\ldots,\theta^k(q)) \mid q\in Q\}$ of 
$\pi \colon \bigoplus_{\alpha=1}^k \T^*Q \rightarrow Q$. 
Note that 
$$\jmath_S^*\omega^\alpha=-\jmath_S^*\d(p^\alpha_i\d y^i)=-\d\theta^\alpha|_S=\pi|_S^*\omega^\alpha_Q\,,\qquad \alpha=1,\ldots,k\,.
$$
Then, $\omega^1_Q,\ldots,\omega^k_Q$ are exact and equivalent to a rank-zero canonical foliated $k$-presymplectic structure. 
\end{proof}
Since there is no way to put $k$ arbitrary closed differential two-forms on~$Q$ into a coordinate system 
so that all of them will have constant coefficients, 
 there will be no general Darboux theorem for foliated $k$-presymplectic manifolds, and thus there is no Darboux theorem for $k$-presymplectic manifolds in general.
 
Theorem \ref{Th::NoGo} can be considered as an extreme case of canonical foliated $k$-presymplectic manifold. For the case of a fibration $\pi|_S:S
\rightarrow Q$ of rank one, it is simple to find new examples where there will be no Darboux coordinates. Assume the simple case of a fibration of rank one given by a submanifold $S\subset \bigoplus_{\alpha=1}^2\cT \mathbb{R}^2$ onto $\mathbb{R}^2$. Since $S$ has dimension three, the two differential forms $\omega_S^1,\omega_S^2$ can be assumed to have rank two and non-trivial common intersection of their kernels. In such a case, they are proportional.  One of them can always be put into canonical form for certain variables, because is presymplectic. Since they are proportional and due to the closeness condition, they depend only on two variables.  Hence, to put them in canonical form with some Darboux variables amounts to putting two different volume forms on $\mathbb{R}^2$ in canonical form for the same Darboux variables, which is impossible. 

\begin{example}\label{Ex::Counter}
Let us describe in more detail a more complex example of a foliated 2-presymplectic manifold that does not admit Darboux coordinates. Consider $\bigoplus_{\alpha=1}^2\cT \mathbb{R}^2$ and the fibration of the submanifold $S$ onto $\mathbb{R}^2$ with rank one of the form
$$
S=\{(p^{(1)}_1(\lambda,y^1,y^2)\d y^1+p^{(1)}_2(\lambda,y^1,y^2)\d y^2,p^{(2)}_1(\lambda,y^1,y^2)\d y^1+p^{(2)}_2(\lambda,y^1,y^2)\d y^2):\lambda,y^1,y^2\in \mathbb{R}\}.
$$
In particular, consider 
$$
p_1^{(1)}=\lambda\,,\qquad p_2^{(1)}=0\,,\qquad p_1^{(2)}=f(\lambda,y^1)\,,\qquad p_2^{(2)}=0\,,
$$
for a certain function  $f(\lambda,y^1)$ such that $\partial f/\partial\lambda$ is different from the constant functions zero and one. Hence, $\omega^1_S=\d\lambda\wedge \d y^1$ and $\displaystyle\omega^2_S=\frac{\partial f}{\partial \lambda} (\lambda,y^1)\d\lambda\wedge \d y^1$,
which are closed, proportional, have rank-one kernel and cannot be put into a canonical form for canonical coordinates because $\omega^1_S,\omega^2_S$ amount to two different volume forms on $\mathbb{R}^2$.      
\end{example}


There are several manners of defining a $k$-presymplectic manifold. The following one offers a possibility.

\begin{definition}
\label{dfn:k-presymplectic-manifold}
Let $M$ be a $(n(k+1)-m)$-dimensional manifold, with $0\leq m\leq nk$. 
A {\it $k$-presymplectic structure} on $M$ is a family 
$(\omega^1,\dotsc, \omega^k,V)$, 
where $V$ is an $r$-dimensional integrable distribution 
and $\omega^1,\dotsc, \omega^k$ are closed differential two-forms on $M$ 
with $\rk\omega^\alpha = 2r_\alpha$ and $r = \sum_{\alpha=1}^k r_\alpha$, 
where $1\leq r_\alpha\leq n$, 
satisfying that
    $$ \restr{\omega^\alpha}{V\times V} = 0\,,\qquad \alpha=1,\ldots, k\,. $$
A manifold $M$ endowed with a $k$-presymplectic structure is called a {\it $k$-presymplectic manifold}.
\end{definition}

We would expect to obtain for every $k$-presymplectic structure $(M,\omega^1,\ldots,\omega^k,V)$, where $\rk \omega^\alpha=2r_\alpha$, with $1\leq r_\alpha\leq n$
, and every $x\in M$ a local coordinate system $\{y^i,p_i^\alpha\}$ around $x$ so that
$$
\omega^\alpha=\d y^{i^\alpha_j}\wedge \d p^{\alpha}_{i^\alpha_j}\,,\qquad \alpha=1,\ldots,k\,,
$$
for certain $i^\alpha_j\in \{1,\ldots,n\}$ for $j=1,\ldots, r_\alpha$ for every $\alpha=1,\ldots,k$. Nevertheless, Example \ref{Ex::Counter} represents a counterexample for the existence of Darboux coordinate system for $k$-presymplectic structures. 

Contrary to previous examples, we will give conditions ensuring that a $k$-presymplectic manifold admits Darboux coordinates. Indeed, the manifold $S$ is three-dimensional, while $k=2$. The associated presymplectic forms have rank two. The distribution $V$ is then two-dimensional and generated, for instance, by the vector fields $\langle \partial/\partial \lambda,\partial/\partial y^2\rangle$. Then, $n$ and $m$ can be fixed to be two and three.  

It is worth noting that for a $k$-presymplectic structure on $M$, any Riemannian metric $g$ on $M$ allows one to obtain a decomposition of a subspace $E\subset \T_xM$ as a direct sum of subspaces
\begin{equation}\label{eq:Decom1}
E^{\kappa_1,\ldots,\kappa_k}=E\cap \left(\bigcap_{\alpha=1}^k (\ker \omega_x^\alpha)^{\kappa_\alpha}\right)\,,
\end{equation}
where $\kappa_\alpha\in \{0,1\}$, while $(\ker \omega_x^\alpha)^0=\ker \omega_x^\alpha$ and  $(\ker \omega_x^\alpha)^1=(\ker \omega_x^\alpha)^{\perp_g}$, where $\perp_g$ is the orthogonal relative to the introduced metric $g$. The main aim of this decomposition is to divide $\T_xM$ into two subspaces, $V,S$, given by direct sums of the subspaces in \eqref{eq:Decom1} in such a manner that  $\omega^\alpha(V)$ and $\omega^\alpha(S)$ have  rank $r_\alpha$, while $\omega^\alpha(V)\cap \omega^\alpha(S)=0$ for $\alpha=1,\ldots,k$. As in the case of $k$-symplectic linear spaces, one can now prove a $k$-presymplectic linear space Darboux theorem. 

\begin{lemma} {\bf ($k$-presymplectic linear Darboux theorem)} Given a  $k$-presymplectic structure
$(\omega^1,\ldots,\omega^k,V)$  on $M$, 
where $\rk\omega^\alpha=2r^\alpha$ for $\alpha=1,\ldots,r$. Let $D=\bigcap_{\alpha=1}^k\ker\omega^\alpha$ have rank $d$ and let $\rk V=r+d = \sum_{\alpha=1}^k r_\alpha+d$ be so that
\begin{equation}\label{eq:Condkpresym}
\dim V_\alpha=r_\alpha,\quad V=D \oplus \bigoplus_{\beta=1}^k V_\beta\,,\quad D+V_\alpha=V\cap \left(\bigcap_{\beta\neq \alpha}\ker \omega^\beta\right)\,,\,\,(k\neq 1)\quad \alpha=1,\ldots,k\ ,
\end{equation}
and
$\dim M=n+r+d$.
Then, at every $\cT_xM$, for $x\in M$, one can set a basis of the form $\{e^1,\ldots,e^n;e^\alpha_{\mu^\alpha_j},v^1,\dotsc,v^d\}$, with $\mu^\alpha_j\in I_\alpha\subset\{1,\dotsc,n\}$ and $\abs{I_\alpha} = r_\alpha$ with $\alpha=1,\ldots,k$, of $\cT_xM$ such that 
$$
\omega_x^\alpha=\sum_{j=1}^{r_\alpha}e^{\mu^\alpha_j}\wedge e^\alpha_{\mu^\alpha_j}\,,\qquad \alpha=1,\ldots,k\,,\qquad D_x=\langle v_1,\ldots,v_d\rangle,\qquad V_{\alpha x}=\left\langle e_\alpha^{\mu^\alpha_j}\right\rangle .
$$
\end{lemma}
\begin{proof} Note that $\dim M=n+r+d$. Since $D=\bigcap_{\alpha=1}^k\ker \omega^\alpha$ has rank $d$, one has that 
\begin{equation}\label{con-pre}
D_x^\circ=\Ima \omega^1_x+\ldots+\Ima\omega^k_x\,,\qquad  \forall x\in M\,,
\end{equation}
is such that $\rk D^\circ =n+r$. 
Since $\omega^\beta|_{V\times V}=0$, it follows that $\omega^\beta(V)\subset V^\circ$ for every $\beta$. We have
$$
\omega^1(V)+\ldots+\omega^k(V)\subset V^\circ.
$$
Note that $\rk V^\circ =n$. From the second and third condition in \eqref{eq:Condkpresym}, it follows that $V_\alpha\cap\ker\omega^\alpha=0$. Moreover, one has that $\rk \omega^\alpha(V_\alpha)=r_\alpha=\rk \omega^\alpha(V)$ for $\alpha=1,\ldots,k$. Consider the supplementary $S=V^{\perp_g}$  to $V$. Then, 
$\rk V^{\perp_g}=\dim M - r-d = n$ and $\rk \omega_x^\beta(S_x)\leq n$ for every $x\in M$.  Due to \eqref{con-pre} and the above, one has that 
$$
\omega^1(V^{\perp_g})+\ldots+\omega^k(V^{\perp_g})
$$
is a distribution of rank $r$ at least.  By our decomposition, every $\alpha$ allows us to divide $V^{\perp_g}$  into two spaces in the form $V^{\perp_g}=\Upsilon_\alpha\oplus \left(\ker \omega^\alpha\cap V^{\perp_g}\right)$, where $\Upsilon_\alpha$ has rank $r_\alpha$ because $\ker \omega^\alpha=n+r+d-2r_\alpha$ and $\ker\omega^\alpha\cap V=r+d-r_\alpha$. Then,  $\omega^\alpha(V^{\perp_g})$ is equal to the image of a subspace of rank $r_\alpha$ of $V^{\perp_g}$ and it therefore has rank $r_\alpha$  and $\omega^\alpha(S)\cap \omega^\alpha(V)=0$. Then,
\begin{equation}\label{eq:MaxDec}
\rank(\omega^1(V^{\perp_g})+\ldots+\omega^k(V^{\perp_g}))=r\,,\quad \omega^1(V)+\ldots+\omega^k(V)= V^\circ\,.
\end{equation}
Note that $\ker \omega^\alpha=n+d+r-2r_\alpha$ and $\omega^\alpha(V)=\omega^\alpha(V_\alpha\oplus(\ker\omega^\alpha \cap V))$. Due to the second expression in \eqref{eq:MaxDec}, the sum of the codistributions $S_*^\alpha=\omega^\alpha(V)$ of $\cT M$ for $\alpha=1,\ldots,k$ has rank $n$, but they do not need to be in direct sum. A non-degenerate contravariant symmetric tensor, $g^*$. on $S_*=S_*^1+\ldots+S_*^k$  can be used to give a decomposition of it into subspaces in direct sum of the form 
$$
S_*^{\kappa_1,\ldots,\kappa_k}=\bigcap_{\alpha=1}^kS^{\kappa_{\alpha}}_{\alpha*},
$$
where $\kappa_\alpha\in \{0,1\}$ for $\alpha=1,\ldots,k$, while $S^1_{\alpha*}=S^{\alpha}_*$ and $S^0_\alpha=(S^\alpha_*)^{\perp_{g^*}}$, namely the orthogonal in $S_*$ of $S_*^\alpha$ relative to $g^*$. 
Take a basis of $S_*$ associated with our decomposition. For the elements of such a basis spanning $S_*^\alpha$, there will be unique elements in $V_\alpha$ whose image under $\omega^\alpha$ give minus the corresponding basis in $S_*^\alpha$. Take a supplementary to $S_*$ in $\cT M$, of dimension $d+r$, dual to a basis adapted to the decomposition of $V$ and vanishing on $V^{\perp_g}$. It is worth noting that we have a decomposition
$$
\T M=\left[\Upsilon_\alpha \oplus (\ker \omega^\alpha \cap V^{\perp_g})\right]\oplus \left[V_\alpha\oplus \left(\bigoplus_{\beta\neq \alpha}V_\beta\right)\oplus D\right]
$$
and a dual one in
$
S_*\oplus (V^{\perp_g})^\circ.  
$
In such a basis, the form of $\omega^\alpha$ goes back to \eqref{eq:Form1} and the same technique in Theorem \ref{thm:k-symp-puntual} gives the canonical form for every $\omega^\alpha$ with $\alpha=1,\ldots,k$. Finally, if $w_1,\ldots,w_d$ is a basis of $D$ dual to the one chosen in $\cT_xM$, one has that
$$
\omega^\beta=\sum_{j=1}^{r_\beta}
e^{\mu_j^\beta}\wedge e^\beta_{\mu_j^\beta}\,,\qquad V_\beta=\left\langle e_\beta^{\mu^\beta_j}\right\rangle,\qquad \beta=1,\ldots,k. 
$$
\end{proof}
As proved above, depending on their exact definition, $k$-presymplectic manifolds need not have a Darboux theorem (whatever this means, because we can have different ways of defining such an object). That is why we hereafter a definition of $k$-presymplectic manifold ensuring the existence of a particular case of $k$-presymplectic Darboux theorem. This is done by assuming the existence of certain integrable distribution with particular properties. 
\begin{definition}
A \textit{$k$-presymplectic manifold} $(M,\omega^1,\ldots,\omega^k,V)$ is a $k$-presymplectic manifold such that $\dim M=n+r+d$ where $d=\rk\bigcap_{\alpha=1}^k\ker \omega^\alpha$ and $\rk \omega^\alpha=2r^\alpha$, the $V$ is an integrable distribution such that $\omega|_{V\times V}=0$  of rank $r+d$ and there are  integrable distributions $\bigoplus_{\alpha=1 }^kV_\alpha$, $V_1,\ldots,V_k, D$ so that
  $$
  V=
  \bigoplus_{\alpha=1}^kV_\alpha \oplus D,\qquad D=\bigcap_{\alpha=1}^k\ker \omega^\alpha,
  \qquad D+V_\beta=\bigcap_{\beta\neq \alpha}\ker \omega^\alpha\cap V\ \  (k\neq 1)\,,\qquad \beta=1,\ldots k\,.$$
\end{definition}
\begin{lemma}
Given a $k$-presymplectic manifold $(M,\omega^1,\ldots,\omega^k, V)$, the distributions $V^\alpha$, with $\alpha=1,\ldots,k$, satisfy that every $x\in M$ admits a coordinated neighbourhood with coordinates 
$$\{y^1,\ldots,y^n,z^1,\ldots,z^d,y^\alpha_{1},\dotsc,y_{r_\alpha}^\alpha\}\,,\qquad \alpha=1,\ldots,k\,,$$
on a neighbourhood of $x$ so that
$$
V_\alpha = 
\left\langle \frac{\partial}{\partial y^\alpha_1},\ldots,\frac{\partial}{\partial y^\alpha_{r_\alpha}} \right\rangle \,,\qquad \alpha=1,\ldots,k\,,\qquad D=\left\langle \frac{\partial}{\partial z^1},\ldots,\frac{\partial}{\partial z^d}\right\rangle\,.
$$
\end{lemma}
\begin{proof} 
Let $y^1,\ldots,y^n$ be common functionally independent first-integrals for 
 all vector fields taking values in $V$. Since $D$ is a regular distribution of rank $d$ given by the intersection of kernels of the closed forms $\omega^1,\ldots,\omega^k$, it is  integrable. It is assumed that $\bigoplus_{\alpha=1}^kV_\alpha$ is integrable. Hence, $V_1\oplus \ldots \oplus V_k$ has common first-integrals $z^1,\ldots,z^d$ such that $\d z^1\wedge \ldots \wedge \d z^d\wedge \d y^1\wedge \ldots\wedge \d y^n\neq 0$.  If $k=1$, the result of our lemma easily follows. Assume that $k>1$. Given different integers $\alpha_1,\ldots,\alpha_{k-1}\in\{1,\ldots,k\}$, one has that, 
$$
V_{\alpha_1}\oplus\dotsb\oplus V_{\alpha_{k-1}}\oplus D=\ker \omega^\beta\cap V\,,
$$
where $\beta$ is the only number in $\{1,\ldots,k\}$ not included in $\{\alpha_1,\ldots,\alpha_{k-1}\}$. 
Hence, the distribution $V_{\alpha_1}\oplus\ldots\oplus V_{\alpha_{k-1}}\oplus D$ has corank $r_\beta$, it is integrable, and the vector fields taking values in it have $r_\beta$ common local first-integrals $y^\beta_1,\ldots,y^\beta_{r_\beta}$ such that 
$$
\d z_1\wedge\ldots\wedge \d z_d\wedge \d y^\beta_1\wedge \ldots \wedge \d y^\beta_{r_\beta}\wedge \d y^1\wedge \ldots\wedge \d y^n\neq 0.
$$
By construction, 
$\{y^1_1,\ldots,y^1_{r_1},\ldots,y^k_1,\ldots,y^k_{r_k},z^1,\ldots,z^d,y^1,\ldots,y^n\}$ becomes a local coordinate system on $M$ and 
$$
V_\alpha=\left(\bigcap_{i=1}^d\ker\d z^i\right)\cap\left(\bigcap_{i=1}^n\ker\d y^i\right)\cap\bigcap_{\substack{\beta\neq \alpha\\i=1,\ldots,r_\beta}}\ker \d y^\beta_i\,.
$$
Moreover, $\dfrac{\partial}{\partial y^\beta_i}$ with $i=1,
\ldots,r_\beta$ vanish on all coordinates $y^{\alpha}_j$ with $\alpha\neq \beta$ and $j=1,\ldots,r_\alpha$.  
Hence, 
$$
\left\langle\frac{\partial}{\partial y^\beta_1},\ldots,\frac{\partial}{\partial y^\beta_{r_\beta}}\right\rangle=V_\beta\,,\qquad \beta=1,\ldots,k\,,
$$
and
$$
\left\langle \frac{
\partial}{\partial z^1},\ldots,\frac{\partial}{\partial z^d}\right\rangle=D\,.
$$
\end{proof}
Once the above is proved, the following theorem is immediate. One only has to slightly adapt Theorem \ref{thm:Darboux k-simp} by considering that $\rk V_\alpha=r_\alpha$ for $\alpha=1,\ldots,k$ and to restrict $\omega^\alpha$  to the integral submanifolds of $V_\alpha\oplus \Upsilon_\alpha$, which have dimension $2r^\alpha$, where $\omega^\alpha$ becomes symplectic.

\begin{theorem}[$k$-presymplectic Darboux theorem]\label{Darboux k-presimp}
Let $(M,\omega^1,\dotsc, \omega^k,V)$ be a $k$-presymplectic manifold such that
$\rank\omega^\alpha=2r_\alpha$, with $1\leq r_\alpha\leq n$. The dimension of $M$ is $n+r+d$.
For every point $x\in M$, there exist local coordinates $\{y^i,y^\alpha_{\mu^\alpha_j},z^j\}$, with $1\leq i\leq n$, $\mu^\alpha_j\in I_\alpha\subseteq\{ 1,\dotsc, n\}$, $\vert I_\alpha\vert=r_\alpha$, $1\leq j\leq r_\alpha$ and $1\leq\alpha\leq k$, such that
\[
\omega^\alpha=\sum_{j=1}^{r_\alpha}\d y^{\mu^\alpha_j}\wedge\d y^\alpha_{\mu^\alpha_j},\quad
V_\alpha=\left\langle\displaystyle\frac{\partial}{\partial y^\alpha_{\mu^\alpha_j}}\right\rangle
\, ,\quad \alpha=1,\ldots,k\,, \quad 
\displaystyle\bigcap_{\alpha=1}^k\ker \omega^\alpha=\left\langle\displaystyle\frac{\partial}{\partial z^j}\right\rangle \,. \]
    \end{theorem}

\section{\texorpdfstring{$k$}--cosymplectic and \texorpdfstring{$k$}--precosymplectic manifolds}\label{sec4:k-cosymp}

Similarly to previous sections, let us study $k$-cosymplectic and $k$-precosymplectic manifolds. Our investigation will introduce  relevant technical issues to be addressed that were not present in previous sections. One of its main differences with respect to previous Darboux theorems relies on the fact that Reeb vector fields are not uniquely defined in the case of $k$-precosymplectic manifolds. This suggests that Darboux coordinates for $k$-precosymplectic manifolds should not assume a canonical form form Reeb vector fields. Moreover, additional conditions will be needed to assume so as to obtain canonical bases for the distributions after having the corresponding differential forms written in a canonical manner.

\begin{definition}\label{deest}
Let $M$ be an $(n(k+1)+k)$-dimensional manifold. A
{\it$k$-cosymplectic structure} in $M$ is a family
$(\eta^\alpha,\omega^\alpha,V)$, with  $1\leq \alpha\leq k$, where  $\eta^1,\ldots,\eta^k$ are closed
one-forms on $M$, while  $\omega^1,\ldots,\omega^k$ are closed two-forms in $M$,
and $V$ is an integrable  $nk$-dimensional integrable distribution in $M$ satisfying that 
\begin{enumerate}[(1)]
\item  
$\eta^1\wedge\dots\wedge \eta^k\neq 0\,$,\quad $\eta^\alpha\vert_V=0\,,\quad \omega^\alpha\vert_{V\times V}=0\,,$
\item  
$\displaystyle {\bigcap_{\alpha=1}^{k}}\left( \ker \eta^\alpha\cap  \ker \omega^\alpha\right)=\{0\}\,$, \quad $\displaystyle \rk \,{\bigcap_{\alpha=1}^{k}}\ker \omega^\alpha=k  \,.$
 \end{enumerate}
A manifold $M$ endowed with a $k$-cosymplectic structure is said to be a 
{\it $k$-cosymplectic manifold}. 
\end{definition}

Every $k$-cosymplectic structure  $(\eta^\alpha ,\omega^\alpha,V)$ in $M$ admits a unique family of vector fields
$R_1,\ldots,R_k$ on $M$, called {\it Reeb vector fields}, such that
\begin{equation}\label{reeb}
    \inn_{R_\alpha}\eta^\beta=\delta^\beta_\alpha\,,\qquad\inn_{R_\alpha}\omega^\beta = 0 \,,\qquad \alpha,\beta=1,\ldots,k\,.
\end{equation}
Note that the existence of Reeb vector fields is independent of the existence or not of the distribution $V$.

Given a one-cosymplectic manifold $(M,\eta,\omega,V)$,  the pair $(\eta,\omega)$ is a special type of cosymplectic structure in $M$ that additionally admits the distribution $V$. Not every cosymplectic structure admits such a $V$. In fact, consider $(M=\mathbb{R}\times \mathbb{S}^2,\eta,\omega)$, where $
\eta$ is the one-form on $M$ obtained by pulling-back the one form $\d t$ on $\mathbb{R}$, and $
\omega$ is the pull-back to $M$ of the standard symplectic form on $\mathbb{S}^2$. Then, $(M=\mathbb{R}\times \mathbb{S}^2,\eta,\omega)$  is not a one-cosymplectic manifold because the standard symplectic form on $\mathbb{S}^2$ does not admit a distribution as commented previously in this paper.

\begin{theorem}[$k$-cosymplectic Darboux theorem]
\label{Darboux k-cosymp}
Given   a $k$-cosymplectic manifold of the form $(M,\eta^1,\ldots,\eta^k,\omega^1,\ldots,\omega^k,V)$,
 every point $x\in M$ admits a neighbourhood with local coordinates 
$
\{x^\alpha,y^i,y^\alpha_i\}
$,
with
$1\leq\alpha\leq k$,
$1\leq i \leq n$,
such that
$$
\eta^\alpha=\d x^\alpha
\,,\qquad
\omega^\alpha=\sum_{i=1}^n\d y^i\wedge\d y^\alpha_i \,,\qquad \alpha=1,\ldots,k. 
$$
In these coordinates, $R_\alpha=\dparder{}{x^\alpha}$ for $\alpha=1,\ldots,k$. If $k\neq 1$, then  $V=\left\langle\dparder{}{y^\alpha_i}\right\rangle$ where $\alpha=1,\ldots,k$ and $i=1,\ldots,n$. If $k=1$ and $[\ker \omega,V]\subset \ker \omega \oplus V$, then  $V=\left\langle\dparder{}{y^\alpha_i}\right\rangle$.
\end{theorem}
\begin{proof} Since $\eta^1,\ldots,\eta^k$ are closed and $\eta^1\wedge\ldots\wedge \eta^k$ does not vanish at any point of $M$, one has that $H=\bigcap_{\alpha=1}^k\ker \eta^\alpha$ is an integrable distribution of corank $k$. Moreover, $V$ is contained in $H$ by the definition of $k$-cosymplectic manifolds. Consider  one of the integral leaves, $\mathcal{S}$, of $H$, and the natural local immersion $\jmath_\mathcal{S}:\mathcal{S}\hookrightarrow M$. The $\jmath_\mathcal{S}^*\omega^\alpha$ along with the restriction of $V$ to $\mathcal{S}$ give rise to a $k$-symplectic manifold since a vector field taking values in $H$ that is orthogonal to $H$ relative to $\omega^1,\ldots,\omega^k$ belongs to $\bigcap_{\alpha=1}^k(\ker\eta^\alpha\cap \ker \omega^\alpha)=0$. Hence, $\jmath_\mathcal{S}^*\omega^1,\ldots, \jmath_\mathcal{S}^*\omega^k$ admit $k$-symplectic Darboux coordinates. Doing the same along different leaves of $H$ and gluing the results, we obtain that $\omega^1,
\ldots,\omega^k,\eta^1,\ldots,\eta^k$ have their canonical form.  Let us explain this in detail. The differential forms $\omega^1,\ldots,\omega^k,\eta^1,\ldots,\eta^k$ are invariant relative to the Reeb vector fields of the $k$-cosymplectic manifold and their value in $M$ can be understood as the extension to $M$ obtained from their value on $\mathcal{S}$ by the extension by one-parametric groups of diffeomorphisms of the vector fields $R_1,\ldots,R_k$.  Consider coordinates $x^1,\ldots,x^k$ rectifying simultaneously the vector fields $R_1,\ldots,R_k$.  If one consider the coordinate system in $M$ given by the coordinates $x^\alpha, y^i,y_i^\alpha$ on $M$, where $y^i,y^\alpha_i$ are invariant under the flows of $R_1,\ldots,R_k$ and match the $k$-symplectic Darboux coordinates on $\mathcal{S}$, one gets that the $x^\alpha$ are functionally independent of the $y^i,y^\alpha_i$. Moreover, since $R_1,\ldots,R_k$ are in the kernels of $\omega^1,\ldots,\omega^k$ and they are invariant relative to $R_1,\ldots,R_k$, it follows that their form on $M$ is the same as in $\mathcal{S}$. Meanwhile, $\eta^\alpha=dx^\alpha$ for $\alpha=1,\ldots,k$. 
 and the forms $\omega^1,\ldots,\omega^k,\eta^1,\ldots,\eta^k$ on $M$ take a canonical form.

To obtain a canonical basis of the distribution $V$, additional conditions must be added for $k=1$. On the other hand, if $k>1$, then  each distribution $V_\alpha$ is the intersection of the kernels of $\omega^\beta$ for $\beta\neq \alpha$ along with the intersection with $\bigcap_{\beta=1}^k \ker \eta^\beta$. They are therefore invariant relative to the Reeb vector fields. So, they can be put in canonical form on $\mathcal{S}$ and extended as previously from $\mathcal{S}$ to vector fields on $M$ with a canonical form. On the other hand, if $k=1$, one has that $V$ may not be the kernel of a closed form invariant relative to the associated Reeb vector field and the previous method fails. To ensure this, one has to assume $[\ker \omega,V]\subset \ker\omega \oplus V$. Equivalently, $[R,V]\subset V$ for the unique Reeb vector field of the one-cosymplectic manifold.  
\end{proof}

The conditions given in \cite[Lemma 5.1.1]{Mer_97} and \cite{LMORS_98} for the Darboux theorem for $k$-cosymplectic manifolds may be a little bit misleading since  a necessary condition in the case $k=1$, namely $V$ must be invariant relative to the action of the Reeb vector field, is not given in \cite[Theorem II.4]{LMORS_98} and \cite[Lemma 5.1.1]{Mer_97}, but in \cite[Remark II.5]{LMORS_98} or \cite[Note 5.2.1]{Mer_97}, respectively,  after them.  Moreover, the above-mentioned condition in \cite[Lemma 5.1.1]{Mer_97,LMORS_98} can be rewritten in a new way, namely  $[R_\alpha,V]\subset V$, with $\alpha=1,\ldots,k$, can be rewritten by saying that the distributions $\ker \omega$ and $V$ are integrable and their direct sum is integrable. This is also commented in \cite{GM_23}.

As shown in the previous theorem, the condition $[\ker \omega, V]\subset V\oplus \ker\omega$ is necessary in order to ensure a canonical form for the elements of a basis of $V$. Notwithstanding, if one is mainly concerned with the canonical form of the $\eta^1,\ldots,\eta^k,\omega^1,\ldots,\omega^k$, this condition can be avoided. This is the reason why we skipped $[\ker \omega, V]\subset V\oplus \ker\omega$ in our definition of $k$-cosymplectic manifolds.

\begin{example}
Let $\{x^1,\ldots,x^k\}$ be a linear coordinate system on $\mathbb{R}^k$. Given the canonical projections $\bar{\pi}_1\colon\R^k\times (\T^1_k)^{*}Q\to\R^k$,
$\bar{\pi}_2\colon\R^k\times (\T^1_k)^{*}Q\to (\T^1_k)^{*}Q$,
$\bar{\pi}_0\colon\R^k\times (\T^1_k)^{*}Q\to \R^k\times Q$.
The canonical model for $k$-cosymplectic structures is 
$$(\R^k\times (\T^1_k)^{*}Q,(\bar{\pi}_1)^*\d x^\alpha,(\bar{\pi}_2)^*\omega^\alpha,V=\ker(\bar\pi_0)_*)\,,
$$
where $\omega^1,\ldots,\omega^k$ are the two-forms of the canonical $k$-symplectic structure on $(\T^1_k)^*Q$.
\end{example}

More generally, one has the following construction.

\begin{example}Let $(N,\varpi^\alpha,\mathcal{V})$ be an arbitrary $k$-symplectic manifold. Given the canonical projections
    \begin{equation*}
        \pi_{\R^k}\colon \R^k\times N\longrightarrow\R^k\,,\qquad \pi_N\colon\R^k\times N\longrightarrow N
    \end{equation*}
define the differential forms
\begin{equation*}
        \eta^\alpha = \pi_{\R^k}^\ast(\d x^\alpha)\,,\quad \omega^\alpha = \pi_N^\ast\varpi^\alpha\,,\qquad \alpha=1,\ldots,k\,.
    \end{equation*}
    The distribution $\mathcal{V}$ in $N$ defines a distribution $V$ in $M=\R^k\times N$ by considering the vector fields on $N$ as vector fields in $M$ in the natural way via the isomorphism $\T M=\T\mathbb{R}^k\oplus \T N$. All conditions given in Definition \ref{deest} are verified, and hence $(M=\R^k\times N,\eta^\alpha,\omega^\alpha,V)$ is a $k$-cosymplectic manifold.
\end{example}

As in the case of $k$-presymplectic manifolds, there are many ways of defining a $k$-precosymplectic structure. Note that in the $k$-precosymplectic case, one cannot, in general, extend the notion of Reeb vector fields to give an object that is uniquely defined. Hence, one may wonder about the necessity of putting them into a canonical form in Darboux coordinates, since they are not unique. Taking this into account, let us give one of the possible definitions for $k$-precosymplectic manifolds. No condition for the determination of the canonical form of the Reeb vector fields will be assumed.

\begin{definition}\label{predeest}
Let $M$ be a manifold of dimension $n(k+1)+k-m$, with $0\leq m\leq nk$. A {\it  $k$-precosymplectic structure} in $M$ is a family $(\eta^\alpha,\omega^\alpha,V)$, with $1\leq \alpha\leq k$, where $\eta^\alpha$ are closed
one-forms in $M$, while $\omega^\alpha$ are closed two-forms in $M$ such that $\rk\omega^\alpha=2r_\alpha$, with $1\leq r_\alpha\leq n$, and $V$ is an integrable distribution in $M$ of corank $n+k$ satisfying that
\begin{enumerate}
\item  
$\eta^1\wedge\dots\wedge \eta^k\neq 0\,,\quad \eta^\alpha\vert_V=0\,,\quad \omega^\alpha\vert_{V\times V}=0\,,\qquad \alpha=1,\ldots,k,$
\item 
$\displaystyle \rk \bigcap_{\alpha=1}^{k}\ker \omega^\alpha= k+d\,,$
\item $\displaystyle \rk\bigcap_{\alpha=1}^{k}\big(\ker \omega^\alpha\cap\ker\eta^\alpha\big) = d\,,$
\item one has that $V$ is an integrable distribution admitting a decomposition into integrable distributions $V=\bigoplus_{\alpha=1}^kV_\alpha\oplus D$ such that $ D+V_\beta=\left(\bigcap_{\beta\neq \alpha}\ker \omega^\alpha
\right)\cap V$ for $ \beta=1,\ldots,k$ and $k\neq 1$ for $\dim V_\beta=r_\beta$.
\end{enumerate}
A manifold $M$ endowed with a $k$-precosymplectic structure is called a 
{\it$k$-precosymplectic manifold}. We here after define $r=\sum_{\alpha=1}^kr_\alpha$.

\end{definition}

\begin{example}
Consider a $k$-presymplectic manifold $(P,\varpi^\alpha,V)$. Let us construct a $k$-precosymplectic structure on $\R^k\times P$. First, consider the canonical projections
$$
\R^k\times P\overset{\pi}{\longrightarrow}P\,,\qquad \R^k\times P\overset{\tau}{\longrightarrow}\R^k\,.
$$ 
Then, define $\eta^\alpha = \tau^\ast\d x^\alpha$, where $x^1,\ldots,x^k$ are linear coordinates in $\R^k$, and  $\omega^\alpha = \pi^\ast\varpi^\alpha$ for $\alpha=1,\ldots,k$. Then, $(\R^k\times P,\eta^\alpha,\omega^\alpha)$ is a $k$-precosymplectic manifold.
\end{example}

Let us prove a technical result that is necessary to asses the role played by the distribution $\bigcap_{\alpha=1}^k\ker \eta^\alpha$ in $k$-precosymplectic manifolds. 

\begin{lemma}
Given a $k$-precosymplectic manifold $(M,\eta^1,\ldots,\eta^k,\omega^1,\ldots,\omega^k, V)$,  every $x\in M$ admits a coordinated neighbourhood with coordinates 
$$\{x^1,\ldots,x^k, y^1,\ldots,y^n,z^1,\ldots,z^d,y^\alpha_{1},\dotsc,y_{r_\alpha}^\alpha\}\,,\qquad \alpha=1,\ldots,k\,,$$
on a neighbourhood of $x$ so that
$$V_\alpha = 
\left\langle \frac{\partial}{\partial y^\alpha_1},\ldots,\frac{\partial}{\partial y^\alpha_{r_\alpha}} \right\rangle \,,\qquad \alpha=1,\ldots,k\,,\qquad D=\left\langle \frac{\partial}{\partial z^1},\ldots,\frac{\partial}{\partial z^d}\right\rangle\,.
$$
\end{lemma}
\begin{proof} 
Since $\eta^1,\ldots,\eta^k$ are closed, they admit potentials  $x^1,\ldots,x^k$, respectively. Let $y^1,\ldots,y^n$ be common functionally independent first integrals for 
 all vector fields taking values in the integrable distribution $V$ such that
 $$
 \d x^1\wedge\ldots\wedge \d x^k\wedge \d y^1\wedge\ldots\wedge \d y^n\neq 0.
 $$ It is assumed that $\bigoplus_{\alpha=1}^kV_\alpha$ is integrable. Hence, $V_1\oplus \ldots \oplus V_k$ has common first integrals $z^1,\ldots,z^d$ such that
 $$\mu=\d x^1\wedge \ldots \wedge \d x^k\wedge \d z^1\wedge \ldots \wedge \d z^d\wedge \d y^1\wedge \ldots\wedge \d y^n\neq 0.
 $$
 Given different integers $\alpha_1,\ldots,\alpha_{k-1}\in\{1,\ldots,k\}$, and $k>1$, one has that, 
$$
V_{\alpha_1}\oplus\dotsb\oplus V_{\alpha_{k-1}}\oplus D=\ker \omega^\beta\cap V\,,
$$
where $\beta$ is the only number in $\{1,\ldots,k\}$ not included in $\{\alpha_1,\ldots,\alpha_{k-1}\}$. 
Hence, the distribution $V_{\alpha_1}\oplus\ldots\oplus V_{\alpha_{k-1}}\oplus D$ has corank $r_\beta$ in $V$, it is integrable, and the vector fields taking values in it have $r_\beta$ common local first-integrals $y^\beta_1,\ldots,y^\beta_{r_\beta}$ such that $\d y^\beta_1\wedge \ldots \wedge \d y^\beta_{r_\beta}\wedge \mu\neq 0$. Note that, if $k=1$, a similar result can be obtained by considering $V=V_1\oplus D$ and some $r_1$ functionally independent integrals of $D$. 
By construction, 
$\{x^1,\ldots, x^k,y^1_1,\ldots,y^1_{r_1},\ldots,y^k_1,\ldots,y^k_{r_k},z^1,\ldots,z^d,y^1,\ldots,y^n\}$ becomes a local coordinate system on $M$ and 
$$
V_\alpha=\left(\bigcap_{i=1}^d\ker\d z^i\right)\cap\left(\bigcap_{i=1}^n\ker\d y^i\right)\cap\Bigg(\bigcap_{\substack{\beta\neq \alpha\\i=1,\ldots,r_\beta}}\ker \d y^\beta_i\,\Bigg)\cap \Bigg(\bigcap_{\beta=1}^k\ker \d x^\beta \Bigg).
$$
for $k>1$. For $k=1$, a similar expression is obtained by skipping the kernels of the $\d y^1_i$. Moreover, $\dfrac{\partial}{\partial y^\beta_i}$ with $i=1,
\ldots,r_\beta$ vanish on all coordinates $y^{\alpha}_j,y^i$, with $\alpha\neq \beta$ and $j=1,\ldots,r_\alpha$, and the $z^1,\ldots,z^d$.  
Hence, 
$$
\left\langle\frac{\partial}{\partial y^\beta_1},\ldots,\frac{\partial}{\partial y^\beta_{r_\beta}}\right\rangle=V_\beta\,,\qquad \beta=1,\ldots,k\,,
$$
and
$$
\left\langle \frac{
\partial}{\partial z^1},\ldots,\frac{\partial}{\partial z^d}\right\rangle=D\,.
$$
\end{proof}
The corresponding Darboux theorem for $k$-precosymplectic manifolds reads as follows.

\begin{theorem}[$k$-precosymplectic Darboux Theorem]\label{Darboux k-precosymp}
Let  $M$ be a $k$-precosymplectic manifold such that $\dim M=n+d+r+k$, while
$\rk\omega^\alpha=2r_\alpha$, with $1\leq r_\alpha\leq n$. 
 Let us assume the existence of $k$ Reeb vector fields $R_1,\ldots,R_k$ spanning an integrable $k$-dimensional distribution and such that they commute among themselves. For every  $x\in M$, there exists a local chart of coordinates 
$$
\{x^\alpha,y^i,y^\alpha_{\mu_\alpha},z^j\}\,,\quad
1\leq\alpha\leq k  \,,\quad  1\leq i\leq n\,,\quad \mu_\alpha\in I_\alpha\subseteq\{1,\dots,n\}\,,\quad \vert I_\alpha\vert = r_\alpha\,,\quad 1\leq j\leq d\,,
$$
 such that
$$
    \eta^\alpha=\d x^\alpha\,,\quad \omega^\alpha=\sum_{\mu_\alpha\in I_\alpha}\d y^{\mu_\alpha}\wedge\d y^\alpha_{\mu_\alpha}\,\qquad \alpha=1,\ldots,k\,.
$$
If additionally $[R_i,V]\subset V$, then
$$
    V=\left\langle\displaystyle\frac{\partial}{\partial y^\alpha_{\mu_\alpha}},    \frac{\partial}{\partial z^j}\right\rangle\,,\qquad  \bigcap_{\alpha=1}^{k}(\ker \eta^\alpha\cap\ker\omega^\alpha)=\left\langle\displaystyle\frac{\partial}{\partial z^j}\right\rangle \,.
$$
\end{theorem}
\begin{proof} Consider the distribution $\Upsilon=\bigcap_{\alpha=1}^k\ker\eta^\alpha$, which is integrable of rank $n+d+r$. One can define a leaf  $S_\lambda$ of $\Upsilon$. Then, one has the immersion $\jmath_\lambda:S_\lambda\hookrightarrow M$. Since $V$ is included in $\Upsilon$, one has that  $\jmath^*_\lambda \omega^1,\ldots,\jmath^*_\lambda \omega^k$ 
 allow us to define a $k$-presymplectic manifold with Darboux coordinates for the $\eta^\alpha$ and the $\omega^\alpha$ which depend smoothly on $\lambda$. Note that the $\omega^\alpha,\eta^\alpha$ are invariant relative to some Reeb vector fields $R_1,\ldots,R_k$ spanning an involutive distribution and commuting among themselves. Using this fact and proceeding as in the Darboux $k$-cosymplectic manifold structure,  we obtain our Darboux coordinates for the $\eta^\alpha$ and the $\omega^\alpha$. Note that the same applies to the canonical basis for $\bigcap_{\alpha=1}^k(\ker \eta^\alpha\cap \ker\omega^\alpha)$ even for $k=1$. 
 
 Notwithstanding, the form of the basis for the distribution $V$ needs the additional condition about its invariance relative to $R_1,\ldots,R_k$. Then, gluing together as in Theorem \ref{Darboux k-cosymp}, the result follows. 
\end{proof}

\begin{remark}
Note that $k$-precosymplectic manifolds admit Reeb vector fields, but they are not uniquely defined by conditions \eqref{reeb}.
One must impose some additional condition on $M$
to determine them uniquely. 
For instance, let us restrict ourselves to a
$k$-precosymplectic structure on 
$\R^k\times M$, where $M$ is a $k$-presymplectic manifold.
Then, if we ask the Reeb vectors fields to be vertical with respect to the projection
$\R^k\times M\to\R^k$, the system of equations
\eqref{reeb} determines univocally Reeb vector fields.
An equivalent way of obtaining this same family 
is taking the vector fields $\displaystyle\left\{\parder{}{x^\alpha} \right\}$ on $\R^k$ and lifting them to $\R^k\times M$
with the trivial connection $\displaystyle\d x^\alpha\otimes\parder{}{x^\alpha}$.
As it is obvious, in Darboux coordinates we have that these vector fields are
$\displaystyle R_\alpha=\parder{}{x^\alpha}$. Note that every $k$-presymplectic structure in this case will also satisfy the conditions established in our Darboux theorem.
\end{remark}


\section{Multisymplectic and premultisymplectic structures}\label{sec5:multisymp}


Let us now comment on certain results on Darboux coordinates for multisymplectic forms \cite{CIL_99,FG_13,RW_19}. First, let us detail some results on (pre)multisymplectic geometry (see \cite{CIL_96,CIL_99,EIMR_12} for further references).

In the context of (pre)multisymplectic geometry, the standard kernel of a differential form is called the {\it one-kernel}.

\begin{definition}
Let $M$ be an $n$-dimensional differentiable manifold.
A closed form $\Omega\in\df^k(M)$ whose one-kernel is a distribution of constant rank is called a {\it premultisymplectic form}. Additionally, if $\inn_X\Omega=0$ for a vector field $X\in \mathfrak{X}(M)$ implies that $X=0$, then $\Omega$ is said to be {\it one-nondegenerate} and it becomes a {\it multisymplectic form}. 
The pair $(M,\Omega)$ is said to be a 
{\it premultisymplectic} or a {\it multisymplectic manifold} of {\it degree $k$}, if the one-kernel of $\Omega$ is one-degenerate or one-nondegenerate, respectively. \end{definition}
	
First examples of multisymplectic manifolds are 
{\it symplectic manifolds}, i.e. 
multisymplectic manifolds of degree $2$, and
{\it orientable manifolds}, namely
multisymplectic manifolds with a {\it volume form}. 

The following is a linear analogue of (pre)multisymplectic manifolds.

\begin{definition}
A $k$-covector $\Omega$ on $\mathbb{R}^n$ is called a {\it premultisymplectic linear form}. If $\inn_v\Omega=0$ for some $x\in \mathbb{R}^n$ implies that $v=0$, then $\Omega$ is said to be {\it one-nondegenerate} and it becomes a {\it multisymplectic linear form} or {\it $k$-plectic linear form}. 
The pair $(\mathbb{R}^n,\Omega)$ is said to be a 
{\it premultisymplectic linear space} or a {\it linear multisymplectic linear space} of {\it degree $k$}, respectively. Multisymplectic linear spaces given by a $k$-covector are also called {\it $k$-plectic vector spaces}. \end{definition}

Other typical examples of multisymplectic manifolds
are given by the so-called {\it bundles of forms}, which, in addition, are the canonical models of multisymplectic manifolds.  These canonical models are constructed as follows.
\begin{itemize}
\item
Let $Q$ be a manifold. Consider the bundle $\rho\colon\Lambda^{k}(\T^*Q)\to Q$, i.e. the {\it  bundle of $k$-forms} in $Q$ (also called the {\it $k$-multicotangent bundle} of $Q$).
This bundle is endowed with a canonical structure called 
the {\it tautological} or {\it canonical form}
$\Theta_Q\in\df^{k}(\Lambda^{k}(\T^*Q))$ given by
 $$
\Theta_{Q_{\widehat{\mu}}}(V_1,\ldots ,V_k)=\inn(\rho_*V_1\wedge\ldots\wedge\rho_* V_k)\widehat{\mu},
 $$
  for every $\widehat{\mu}\in\Lambda^k(\T^*Q)$
 and $V_1,\ldots,V_k\in\T_{\widehat{\mu}}(\Lambda^k(\T^*Q))$. 
Then, $\Omega_Q=\d\Theta_Q\in\df^{k+1}(\Lambda^k(\T^*Q))$
is a one-nondegenerate form and hence
$(\Lambda^k(\T^*Q),\Omega_Q)$ is a multisymplectic manifold of degree $k+1$.
Furthermore, denoting by $\{x^i,p_{i_1\ldots i_k}\}$ the charts of natural coordinates in $\Lambda^k(\T^*Q)$,
these canonical forms read locally as
$$
\Theta_Q=p_{i_1\ldots i_k}\d x^{i_1}\wedge\ldots\wedge\d x^{i_k}
\,, \quad
\Omega_Q=\d p_{i_1\ldots i_{k}}\wedge\d x^{i_1}\wedge\ldots\wedge\d x^{i_{k}} \ .
$$
Such coordinates are {\it Darboux coordinates} in $\Lambda^k(\T^*Q)$.
\item
If $\pi\colon Q\to M$ is a fibre bundle,
let $\rho_r\colon\Lambda^k_r(\T^*Q)\to Q$ be the subbundle of $\Lambda^k(\T^*Q)$
made of the $r$-{\it horizontal $k$-forms} in $Q$ with respect to the projection $\pi$, namely the $k$-forms in $Q$ vanishing when applied to $r$ vector fields that in $Q$ that are $\pi$-vertical.
If $\rho^k_r\colon\Lambda_r^k(\T^*Q)\to \Lambda^k(\T^*Q)$ is the canonical injection,
then $\Theta^r_Q=(\rho^k_r)^*\Theta_Q\in\df^k(\Lambda^k_r(\T^*Q))$ is the tautological $k$-form in $\Lambda^k_r(\T^*Q)$, and then,
taking $\Omega^r_Q=\d\Theta^r_Q\in\df^{k+1}(\Lambda^k_r(\T^*Q))$,
we have that
$(\Lambda^k_r(\T^*Q),\Omega^r_Q)$ is a multisymplectic manifold of degree $k+1$.
As above, the charts of natural coordinates in $\Lambda^k_r(\T^*Q)$
are also charts of Darboux coordinates, on which these canonical forms have local expressions similar to the above ones. 
\end{itemize}

Nevertheless, in general, multisymplectic manifolds are not (locally)
diffeomorphic to these canonical models.

Note that a multisymplectic form with Darboux coordinates admits a local flat connection compatible with it. Furthermore, if the multisymplectic form has a compatible flat connection, it admits coordinates in which the multisymplectic form has constant coordinates, but it does not need to be of the previous form. In particular, if a multisymplectic form has kernels of higher order to those of $\Omega_Q$, then there is no Darboux theorem in the above senses. In particular, this is a typical problem for Darboux coordinates: differential forms can be put into a form with constants coefficients in many manners and Darboux theorems use to stress one particular form over others, although others may be of interest too.

In general, multisymplectic manifolds do not need to have a coordinate system that makes the multisymplectic form to have constant coordinates, which is the very most basic condition for the existence of a Darboux theorem. Indeed, multisymplectic manifolds of this type are called {\it flat} in the literature \cite{RW_19}. The exact definition is given next.

\begin{definition} A multisymplectic manifold $(M, \omega)$ is called \textit{flat near} $x\in M$, if there exists a mapping $\phi : U\subset M \rightarrow \T_xM$ such that $\phi(x) = 0$ and $\phi^*\omega_x=\omega$ for $\omega_x$ being a constant-coefficient non-degenerate multilinear-form on $\T_xM$. 

\end{definition}

\begin{definition}
An $(n+1)$-plectic vector space $(V, \omega)$ is called \textit{standard} if there exists a linear subspace $W\subset V$ such that $\inn_{u\wedge v}\omega= 0$ for all $u, v \in W$, and
$$
    \omega^\sharp: w\in W\mapsto \omega^\sharp(w)\in  \Lambda^n(V/W)^*
$$
such that $\omega^\sharp(w)(v_1 + W, \dotsc, v_n + W)=\omega(w, v_1, \dotsc, v_n)$ for every $v_1,\ldots,v_n\in V$, 
is an isomorphism.
\end{definition}

 In the above situation, $W$ is unique if $n\geq 2$ and then often denoted $W_\omega$.
From 
\cite{CIL_99,Mar_88,RW_19}, the following result can easily be derived.

\begin{theorem}
Let $n \geq 2$ and let $(M, \omega)$ be a standard $(n+1)$-plectic manifold, i.e. $(M, \omega)$ has as constant
linear type fixed standard $(n+1)$-plectic vector space. Then, $W_\omega =\bigsqcup_{x\in M}W_{\omega_x}\subset TM$ is a smooth distribution. Furthermore, $(M, \omega)$ is flat if and only if $W_\omega$ is integrable.        
\end{theorem}

Let us just recall that our definition of $(n+1)$-plectic symplectic manifold is sometimes called a $n$-plectic manifold in the literature \cite{RW_19}.

Let us now turn to a type of multisymplectic manifold for which we will obtain Darboux coordinates. 

\begin{definition}
A {\it special multisymplectic manifold} is a multisymplectic
manifold $(M,\Omega)$ of degree $k$ such that
$\Omega=\d\Theta$, for some $\Theta\in\df^{k-1}(M)$, and
there is a diffeomorphism $\phi\colon M\to \Lambda^{k-1}(\T^*Q)$,
$\dim\, Q=n\geq k-1$,
(or $\phi\colon M\to \Lambda^{k-1}_r(\T^*Q)$),
and a fibration $\pi\colon M\to Q$
such that $\rho\circ\phi=\pi$
(resp. $\rho_r\circ\phi=\pi$),
and $\phi^*\Theta_Q=\Theta$ (resp. $\phi^*\Theta_Q^r=\Theta$).
\end{definition}

And, as a result of the above discussion, we state the following result.

\begin{theorem}[Restricted multisymplectic Darboux Theorem]
Special multisymplectic manifolds $(M,\Omega)$ are {\it multisymplectomorphic} to bundles of forms.
Therefore, there is a local chart of {\it Darboux coordinates}
 around every point $x\in M$.
\end{theorem}

Like in the $k$-symplectic and $k$-cosymplectic cases, some additional properties are needed
to assure the existence of Darboux-type coordinates \cite{Mar_88}
and then to have multisymplectic manifolds 
that locally behave as the canonical models.
To state these additional conditions, we need to introduce some generalisations of concepts of symplectic geometry.
So, if $(M,\Omega)$ is a multisymplectic manifold of degree $k$
and ${\mathcal W}$ a distribution on $M$
, we define \cite{CIL_99,LMS_03}
the {\it $r$-orthogonal multisymplectic vector space} at $p\in M$ of $\mathcal{W}$ as
$$
{\mathcal W}_x^{\perp,r}=\{ v\in\T_xM\mid
\inn(v\wedge w_1\wedge\ldots\wedge w_r)\Omega_p=0\,,\
\forall w_1,\ldots,w_r\in{\mathcal W}_x\} \,.
$$
Then, the {\it $r$-orthogonal multisymplectic complement} of ${\cal W}$
is the distribution
${\mathcal W}^{\perp,r}=\displaystyle\bigsqcup_{x\in M}{\mathcal W}_x^{\perp,r}$,
and we say that
${\mathcal W}$ is an {\it $r$-coisotropic} or an {\it $r$-isotropic  distribution} if
\ ${\mathcal W}^{\perp,r}\subset{\mathcal W}$ or
\ ${\mathcal W}\subset{\mathcal W}^{\perp,r}$, respectively
(if \ ${\mathcal W}={\mathcal W}^{\perp,r}$
then ${\mathcal W}$ is an {\it $r$-Lagrangian distribution}). Let us use previous notions. 

\begin{definition}
\label{special-multi}
Let $(M,\Omega)$ be a multisymplectic manifold of degree $k$,
and let  ${\mathcal W}$ be a regular one-isotropic involutive distribution in $(M,\Omega)$.
\begin{enumerate}
\item
A {\it multisymplectic manifold of type $(k,0)$} is a triple 
$(M,\Omega,{\mathcal W})$ such that,
for every $x\in M$,
\begin{enumerate}
\item
$\dim {\mathcal W}(x)=\dim \Lambda^{k-1}(\T_xM/{\mathcal W}(x))^*$.
\item
$\dim (\T_xM/{\mathcal W}(x))>k-1$.
\end{enumerate}
\item
A {\it  multisymplectic manifold of type $(k,r)$}
($1\leq r\leq k-1$) is a quadruple
$(M,\Omega,{\mathcal W},{\mathcal E})$,
where ${\mathcal E}$ is a distribution on $M$ such that, for every $x\in M$, one has that
${\mathcal E}(x)$ is a vector subspace of $\T_xM/{\mathcal W}(x)$
satisfying the following properties:
\begin{enumerate}
\item
If $\pi_x\colon\T_xM\to\T_xM/{\mathcal W}(x)$ is
the canonical projection, then
$\inn(v_1\wedge\ldots\wedge v_r)\Omega_x=0$, for every $v_i\in\T_xM$
such that $\pi_x(v_i)\in{\mathcal E}(x)$ ($i=1,\ldots,r$).
\item
$\dim {\mathcal W}(x)=\dim \Lambda_r^{k-1}(\T_xM/{\mathcal W}(x))^*$,
where the horizontal forms are considered with respect to the
subspace ${\mathcal E}(x)$.
\item
$\dim (\T_xM/{\mathcal W}(x))>k-1$.
\end{enumerate}
\end{enumerate}
\end{definition}

Then, the fundamental result is the following \cite[Corollary 3.31]{LMS_03}.

\begin{theorem}[Generalised multisymplectic Darboux Theorem]
Every multisymplectic manifold $(M,\Omega)$ of type $(k,0)$ 
(resp. of type $(k,r)$)
is locally multisymplectomorphic to a bundle of $(k-1)$-forms
$\Lambda^{k-1}(\T^*Q)$ (resp. $\Lambda^{k-1}_r(\T^*Q)$),
for some manifold $Q$; that is, to a canonical multisymplectic manifold.
Therefore, there is a local chart of {\it Darboux coordinates}
 around every point $x\in M$.
\end{theorem}

\begin{definition}
Multisymplectic manifolds that are locally multisymplectomorphic
to bundles of forms are called 
{\it  locally special multisymplectic manifolds}.
\end{definition}

As a relevant example, if $\pi\colon E\to M$ is a fiber bundle 
(where $M$ is an $m$-dimensional oriented manifold), 
$J^1\pi$ is the corresponding
first-order jet bundle, and ${\cal L}$ is a first-order regular or hyperregular 
Lagrangian density, then the Poincar\'e--Cartan form 
$\Omega_{\cal L}\in\df^{m+1}(J^1\pi)$
is a multisymplectic form and $(J^1\pi,\Omega_{\cal L})$ is a
(locally) special multisymplectic manifold.
If ${\cal L}$ is a singular Lagrangian, then $(J^1\pi,\Omega_{\cal L})$ is a premultisymplectic manifold.

\begin{definition}
A {\it special premultisymplectic manifold} is a premultisymplectic
manifold $(M,\Omega)$ of degree $k$ such that $M/\ker \Omega$ is a manifold and the unique multisymplectic form $\Omega'$ on $M/\ker \Omega$ such that $\pi^*\Omega'=\Omega$ is a special multisymplectic form.
\end{definition}

The following naturally follows.

\begin{definition}
\label{special-premulti}
Let $(M,\Omega)$ be a premultisymplectic manifold of degree $k$,
and ${\mathcal W}$ a regular one-isotropic involutive distribution in $(M,\Omega)$ such that $\ker \Omega\subset \mathcal{W}$ and $d=\dim \ker\Omega$.
\begin{enumerate}
\item
A {\it premultisymplectic manifold of type $(d,k,0)$} is a triple 
$(M,\Omega,{\mathcal W})$ such that,
for every $x\in M$,
\begin{enumerate}
\item
$\dim {\mathcal W}(x)-d=\dim \Lambda^{k-1}(\T_xM/{\mathcal W}(x))^*$.
\item
$\dim (\T_xM/{\mathcal W}(x))>k-1$.
\end{enumerate}
\item
A {\it premultisymplectic manifold of type $(d,k,r)$}
($1\leq r\leq k-1$) is a quadruple
$(M,\Omega,{\mathcal W},{\mathcal E})$,
where ${\mathcal E}$ is a distribution on $M$ such that, for every $x\in M$, the  space
${\mathcal E}(x)$ is a vector subspace of $\T_xM/{\mathcal W}(x)$
with the following properties:
\begin{enumerate}
\item
If $\pi_x\colon\T_xM\to\T_xM/{\mathcal W}(x)$ is
the canonical projection, then
$\inn(v_1\wedge\ldots\wedge v_r)\Omega_p=0$, for every $v_i\in\T_xM$
such that $\pi_x(v_i)\in{\mathcal E}(x)$, $i=1,\ldots,r$.
\item
$\dim {\mathcal W}(x)-d=\dim \Lambda_r^{k-1}(\T_xM/{\mathcal W}(x))^*$,
where the horizontal forms are considered with respect to the
subspace ${\mathcal E}(x)$.
\item
$\dim (\T_xM/{\mathcal W}(x))>k-1$.
\end{enumerate}
\end{enumerate}
\end{definition}


\begin{theorem}[Generalised premultisymplectic Darboux Theorem]
Every premultisymplectic manifold $(M,\Omega)$ of type $(d,k,0)$ 
(resp. of type $(d,k,r)$)
is locally premultisymplectomorphic to a canonical premultisymplectic manifold of type $(d,k,0)$ (resp. of type $(d,k,r)$).
Therefore, there is a local chart of {\it Darboux coordinates}
 around every point $x\in M$.
\end{theorem}

As in previous structures, analogous claims can be done concerning the existence of connected compatible connections with premultisymplectic manifolds.


\section{Conclusions and outlook}\label{sec:outlook}

The focus of this research is the exploration of Darboux-type theorems concerning geometric structures defined by closed differential forms.
The initial section of this study entails an examination of the Darboux theorem for symplectic, presymplectic, and cosymplectic manifolds.
By imposing minimal regularity conditions, we have successfully established a proof for a Darboux theorem applicable to precosymplectic manifolds. Within the realm of geometric mechanics, these manifolds serve as the phase spaces for both regular and singular autonomous and non-autonomous dynamical systems.

We have presented novel proofs for the Darboux theorem concerning $k$-symplectic and $k$-cosymplectic manifolds. These proofs appear to be simpler compared to the previously known ones. Additionally, we have introduced and demonstrated new Darboux theorems for specific families of $k$-presymplectic and $k$-precosymplectic manifolds. Furthermore, we have provided a counterexample illustrating that a general Darboux-type theorem does not hold for $k$-presymplectic manifolds. We have conducted a thorough review of previous findings regarding the existence of Darboux coordinates for certain types of multisymplectic manifolds. Lastly, we have presented fresh results that establish the existence of Darboux coordinates for particular cases of premultisymplectic manifolds. All of these structures play a vital role in the geometric representation of both regular and singular classical field theories. The relations of Darboux theorems with flat connections have been studied, which provides new viewpoints and gathers previous scattered results in the literature.

The ideas of this paper can be extended to other geometric structures related with closed one or two-forms of different types. 
Notwithstanding, the formalism on flat compatible connections does not apply to geometric structures related to families of different forms that do not allow for a locally constant form and therefore closed, e.g. for contact and precontact structures, and their extensions
(which appear, for instance, in the geometric description of dissipative and action-dependent systems in physics). It would be interesting to find an analogue of our formalism for such theories. In particular, note that non-closed differential forms may have flat compatible connections provided torsion different from zero is allowed. For instance, consider the manifold $M = \R^3$ with natural coordinates $\{t,x,p\}$, the one-form $\eta = \d t - p\d x$ and the connection $\nabla$ in $M$ whose only non-vanishing Christoffel symbol is $\Gamma_{px}^t = -1$. It is easy to check that $\eta$ is a contact one-form on $M$, and parallel relative to the connection $\nabla$, namely $\nabla\eta = 0$. However, the connection $\nabla$ is not torsion-free: its torsion has local expression $T = \d x\otimes\d p\otimes\dparder{}{t} - \d p\otimes\d x\otimes\dparder{}{t}$. This torsion takes account of the non-integrability of the contact distribution $D = \ker\eta$. Meanwhile, $\nabla$ is flat. The relation between integrability of a geometric structure and the torsion of compatible connections will be investigated in a future work.

Moreover, this work has studied conditions for Darboux theorems of various types. We believe that there is still room to provide more types of Darboux coordinates, and that more research in the study of necessary and sufficient conditions for their existence is needed. In particular, this specially applies to $k$-pre(co)symplectic manifolds. 
\subsection*{Acknowledgments}
We thank M. de León and J. Gaset for fruitful discussions and comments.
We acknowledge partial financial support from the 
{\it Spanish Ministry of Science and Innovation}, grants PID2021-125515NB-C21, PID2021-125515NB-C22, and RED2022-134301-T of AEI, and of the {\it Ministry of Research and Universities of
the Catalan Government}, project 2021 SGR 00603 {\sl Geometry of Manifolds and Applications, GEOMVAP}.
J. de Lucas and X. Rivas acknowledge partial financial support from project IDUB with number PSP: 501-D111-20-2004310. X. Rivas would like to thank the cordiality shown during his stays at the Faculty of Physics of the University of Warsaw financed from the above mentioned IDUB project.

\bibliographystyle{abbrv}
\itemsep 0pt plus 1pt
\small
\bibliography{references.bib}





\end{document}